\documentclass{amsart}

\usepackage{amsmath}
\usepackage{amsthm}
\usepackage[foot]{amsaddr}
\usepackage{amsfonts}
\usepackage{amssymb}
\usepackage{graphics}
\usepackage[mathscr]{euscript}
\usepackage{listings}
\usepackage{enumitem}
\usepackage{multirow}
\usepackage{mathtools}
\usepackage{xcolor}
\usepackage{setspace}
\usepackage[colorlinks=true, linkcolor=red, citecolor=blue]{hyperref}

\newcommand{\bb}[1]{\ensuremath{\mathbb{#1}}}

\newcommand{\ca}[1]{\ensuremath{\mathcal{#1}}}
\newcommand{\E}[1]{\ensuremath{\mathbf{E}\left[#1\right]}}
\newcommand{\Var}[1]{\ensuremath{\mathbf{Var}\left[#1\right]}}
\newcommand{\Cov}[1]{\ensuremath{\mathbf{Cov}\left[#1\right]}}
\newcommand{\abs}[1]{\ensuremath{\left| #1 \right|}}
\DeclarePairedDelimiter{\card}{|}{|}

\newcommand{\R}{\mathbb{R}}

\newcommand{\N}{\mathbb{N}}
\newcommand{\Z}{\mathbb{Z}}

\newcommand{\eps}{{\varepsilon}}

\newcommand{\Bin}{\ensuremath{\operatorname{Bin}}}
\newcommand{\Ber}{\ensuremath{\operatorname{Ber}}}

\newcommand{\defeq}{:=}
\newcommand{\adj}{\sim}

\newcommand{\one}{\mathbf{1}}
\newcommand{\Neigh}{\Gamma}

\renewcommand{\epsilon}{\eps}
\renewcommand{\phi}{\varphi}
\renewcommand{\iff}{\Leftrightarrow}
\renewcommand{\Pr}{{\mathbf P}}

\theoremstyle{plain}
\newtheorem{theorem}{Theorem}[section]
\newtheorem{conjecture}[theorem]{Conjecture}

\newtheorem{proposition}[theorem]{Proposition}

\newtheorem{lemma}[theorem]{Lemma}
\newtheorem{corollary}[theorem]{Corollary}

\newtheorem{question}[theorem]{Question}

\theoremstyle{definition}
\newtheorem{definition}[theorem]{Definition}

\newtheorem{remark}[theorem]{Remark}
\newtheorem{experiment}[theorem]{Experiment}

\allowdisplaybreaks

\newcommand{\dif}{\delta}

\newcommand{\pup}{p_{\textup{up}}}

\newcommand{\pac}{p_{\textup{ac}}}

\newcommand{\adv}{\Delta}
\newcommand{\Cbe}{C_{\scriptscriptstyle\textup{BE}}}

\newcommand{\upd}[1]{#1^{\textup{up}}}

\newcommand{\ac}[1]{#1^{\textup{ac}}}

\newcommand{\hypo}[1]{\widehat{#1}}

\usepackage{latexsym}
\newcommand{\obs}{\leadsto}
\newcommand{\Clr}{\ell}

\theoremstyle{plain}

\lstdefinestyle{myStyle}{
	language=Python,
	basicstyle=\footnotesize\tt,
	keywordstyle=\color{orange!70!black},
	commentstyle=\itshape\color{gray!60!black},
	stringstyle=\color{green!50!black},
	stepnumber=1,
	tabsize=2,
	numbers=none,
	numberstyle=\tiny,
	numbersep=5pt,
	showspaces=false,
	escapechar=`,
	showstringspaces=false}
\title[The power of few]{The ``Power of Few'' phenomenon: the sparse case}
\date{\today}
\author[BaoLinh Tran]{BaoLinh Tran}
\email{l.tran@yale.edu}
\address{Dept. of Mathematics, Yale University, 219 Prospect St, Floor 9, New Haven, CT 06511}
\author[Van Vu]{Van Vu}
\email{van.vu@yale.edu}
\begin{document}
	
\begin{abstract}
		The ``majority dynamics'' process on a social network begins with an initial phase, where the individuals are split into two competing parties, Red and Blue. Every day, everyone updates their affiliation to match the majority among those of their friends.
		While studying this process on
        Erd\"{o}s-R\'{e}nyi
        $G(n, p)$ random graph (with constant density), the authors discovered  the ``Power of Few''  phenomenon, showing that a very small advantage to one side already guarantees that  everybody will unanimously join that side after just a few  days with overwhelming probability.
		For example, when $p = 1/2$, then 10 extra members guarantee this unanimity with a 90\% chance, regardless of the value of $n$. 
		
		In this paper, we study  this phenomenon for sparse random graphs. 
      It is clear that below the connectivity threshold, the phenomenon ceases to hold, as the isolated vertices never change their colors. We show that 
      it holds for every density above the threshold. 

  To make the process more realistic, we  
  also assume that individuals can randomly activate their accounts to post their opinions and observe their neighbors (just as we do on social media). We prove that the phenomenon is robust under this assumption. To motivate further research, we also state a number of open questions with supportive  numerical experiments. 
	\end{abstract}

\maketitle

\newpage

\section{Introduction} \label{sec:introduction}
\subsection {Majority Dynamics}

The study of processes on a social network where individuals form and change their opinions based on those around them  (opinion dynamics) is an  active research area with applications in various fields. 
The 2017 survey \cite{mossel2017} by Mossel and Tamuz discussed applications in Biology
, Economics
, Statistical Physics
. More recent applications have been found in Computer Science, particularly in network security \cite{alistarh2015,alistarh2017,zehmakan2019,gacs1978,land1995}.

In this paper, we  focus on a process called \emph{majority dynamics}, where each individual starts with a color (Red or Blue, representing different opinions), and updates their color each day to match the one the majority of their friends hold on the previous day.
\begin{definition}[Majority Dynamics]  \label{defn:majority-dynamics}
	Given simple undirected graphs $G = (V, E)$, and a partition $V = R\cup B$, the \textbf{majority dynamics} on $G$ with initial coloring $(R, B)$ is a process that begins with Day 0, where vertices in $R$ and $B$ are respectively colored Red and Blue.
 For each $t \ge 0$, on day $t + 1$,
    each vertex adopts the color the majority of its neighbors hold on day $t$, or remains unchanged in case of a tie.
\end{definition}

\subsection{Unanimity and the ``Power of Few'' phenomenon}

How the process terminates depends strongly on the structure of the network. 
In this paper, we consider
Erd\"{o}s-R\'{e}nyi
$G(n, p)$ random graphs and focus on  the question of \emph{unanimity}: Given an initial coloring, what is the probability for each color to \emph{win}, meaning it covers the whole graph at some point?

This natural question has been  tackled in many recent works by various authors \cite{benjamini,fountoulakis,zehmakan2019,devlinberkowitz2022,sahsawhney2021,jhkim2021}.

Consider the balanced case, when both color classes begin with $n/2$ vertices, then it is clear that the winning probabilities (of Red and Blue) are equal, by symmetry. 
When an initial bias, or \emph{gap} $\adv > 0$, is introduced (say Red begins with $n/2 + \adv$ vertices), the winning odds shift in favor of Red. But how fast?

In a previous paper, the authors discovered a surprising  ``Power of Few'' phenomenon: when $p$ is a constant, $\adv$ only needs to be a constant for Red to win within 4 days with probability arbitrarily close to $1$ \cite{TranVuDense2020}, regardless of the size $n$ of the population. For instance, when $p = 1/2$, $\adv = 5$ already guarantees that Red wins with a probability of at least $.9$ \cite{TranVuDense2020}.
It is thus natural to ask: can we extend this phenomenon to smaller values of $p$?

It is clear that when $p$ is below the connectivity threshold, $n^{-1}\log n$, the phenomenon fails, since there are, with probability $1 - o(1)$, isolated vertices that never change their colors.
Under what assumption can we ensure that the threshold for connectivity is also the threshold for unanimity?

\subsection{Main result}

Our first main theorem
establishes an assumption guaranteeing unanimity for every $p$ above the connectivity threshold, resolving this question.

\begin{theorem}[The Power of Few for sparse graphs] \label{thm:main-nonlazy}
	Consider a set $V$ of $n$ vertices with an arbitrary partition $V = R\cup B$ where $\card{R} = n/2 + \Delta$, then draw a $G(n, p)$ random graph over $V$.
	If $n - 10 \ge pn \ge (1 + \lambda)\log n$ and $p\adv \ge 10$, where $\lambda > 0$ is a constant, then majority dynamics on $G$ with initial coloring $(R, B)$ achieves a Red unanimity within time $O(\log_{pn}n)$ with probability at least
	\[
	1 - O\left(\frac{1}{\min\{n, p\adv^2\}} + n^{-\lambda/2}\right).
	\]
\end{theorem}

The key point here is that $\adv$ only depends directly on $p$ and not $n$.
In the special case when $p$ is a constant, $\adv$ needs only to be a constant for Red to win with probability arbitrarily close to $1$, thus this theorem extends the ``Power of Few'' phenomenon in \cite{TranVuDense2020}.

\subsubsection*{Notation for asymptotics}
Throughout the paper, we write $a = O(b)$ for two
quantities $a, b > 0$ that depend on the same set of variables, which includes the variable $m$,
if $a \le Cb$ for a universal constant $C$, and $a = O_m(b)$ if $C = C_m$ only depends on $m$.
We write $a = o_m(b)$ if $a \le c_m b$, where $c_m$ only depends on $m$ and $\lim_{m\to \infty} c_m = 0$.
Note that the second asymptotics has to involve a variable that grows to $+\infty$, however, since
the order of the graph $n$ plays this role most of the time, we often use $a = o(b)$ in place of $a = o_n(b)$.
This semantics apply analogously for $a = \Omega(b)$, $a = \Omega_m(b)$, $a = \Theta(b)$, $a = \Theta_m(b)$, $a = \omega_m(b)$ and $a = \omega(b)$.

\begin{remark}
	Note that, while  we assume the initial coloring is fixed before the graph is drawn, we can study the process in the following  equivalent models:
	
	\textbf{Model 1.} First draw a $G(n, p)$ random graph over $V$, then uniformly sample a $(n/2 + \adv)$-subset of $V$ to color Red and color the rest Blue.

    \textbf{Model 2.} (For even $n$) First fix an arbitrary even partition $(R, B)$ of $V$, namely one such that $|R| = |B|$,
    then draw a $G(n, p)$ random graph over $V$, finally uniformly sample a $\adv$-subset of ``defectors'' from $R$ and move them to $B$.
	
	\textbf{Model 3.} (For even $n$) First draw a $G(n, p)$ random graph over $V$, then uniformly sample an even partition $(R, B)$ of V and color them accordingly, finally uniformly sample a $\adv$-subset of ``defectors'' from $R$ and move them to $B$.
\end{remark}

In the next three subsections, we will discuss the applications of Theorem \ref{thm:main-nonlazy} in proving and strengthening some recent results on unanimity, its generalization to a more realistic setting, and some open questions and conjectures.

\subsection{Relations with the random coloring scheme}

Another process considered in the literature 
involves coloring every vertex of a random graph Red or Blue independently with probability $1/2$ each \cite{benjamini,fountoulakis,jhkim2021}, and then run the majority dynamics. 
We  refer to this as the \textbf{random $1/2$ coloring scheme}.
Benjamini, Chan, O'Donnell, Tamuz and Tan  \cite{benjamini} proved that in this coloring scheme, the side with the initial majority wins with probability at least $.4$ when $p\ge \lambda n^{-1/2}$, where $\lambda$ is a  sufficiently large constant. 
Fountoulakis, Kang and Makai  \cite{fountoulakis} improved the  winning probability to  $1 - \eps$ for any constant $\eps > 0$,  provided that $\lambda = \lambda_\eps \xrightarrow{\eps \to 0} +\infty$, and unanimity is achieved only in 4 days.


We can use Theorem \ref{thm:main-nonlazy} to give a quick proof for this result,  with some ``delay''.
\begin{theorem}
	For any fixed  $\eps > 0$, there is a constant $\lambda > 0$ such that: in majority dynamics on a $G(n, p)$ random graph with $p \ge \lambda n^{-1/2}$ and a random $1/2$ initial coloring, the larger side wins in $O(1)$ days with probability at least $1 - \eps$.
\end{theorem}
\begin{proof}
	Let $R$ and $B$ respectively be the Red and Blue camps after the initial coloring and $\adv$ the initial gap.
	The quantity $2\adv = ||R| - |B||$ is a sum of independent Rademacher variables, with mean $0$ and variance $n$, so by the central limit theorem (see Corollary \ref{lem:strongesseen}),  for any $\eps > 0$, $\adv \ge C_\eps\sqrt{n}$ for some constant $C_\eps > 0$ with probability at least $1 - \eps/2$.
	Conditioned on a fixed instance of $(R, B)$ satisfying this, Therem \ref{thm:main-nonlazy} implies that if $p \ge 10\adv^{-1}$ and $p\adv^2 \xrightarrow{n \to \infty} \infty$, the majority side wins in $O(\log_{pn}n)$ days with probability $1 - o(1)$.
	Both conditions are satisfied when $p \ge (10/C_\eps)n^{-1/2}$. The winning time is $O(\log_{pn}n) = O(1)$, so by choosing $\lambda_\eps = 10/C_\eps$ we complete the proof.
\end{proof}

A careful inspection of Theorem \ref{thm:main-nonlazy}'s proof reveals that the winning time can be taken to be at most 5 days, which is very close to the original theorem in \cite{fountoulakis}.

If one considers a more general \emph{random $(p_r, p_b)$ coloring scheme}, where the probability for a vertex to be Red is $p_r$ and Blue $p_b = 1 - p_r$, then the initial gap is larger for which unanimity can be  proven for lower $p$.  Zehmakan in 2019 \cite{zehmakan2019} reached the threshold $p = (1 + \lambda)n^{-1}\log n$ for such a biased random coloring, proving that Red wins with probability $1 - o(1)$ if $p_r = 1/2 + \omega((pn)^{-1/2})$, i.e. Red has a bias $\omega((pn)^{-1/2})$. Our theorem \ref{thm:main-nonlazy} implies  that a  smaller bias 
$\omega((pn)^{-1})$ suffices.

\begin{theorem}
	On a $G(n, p)$ random graph where $p \ge (1 + \lambda)n^{-1}\log n$ for a constant $\lambda > 0$, majority dynamics with a random $(p_r, p_b)$ initial coloring with $p_r = 1/2 + \omega((pn)^{-1})$ terminates with a Red unanimity with probability $1 - o(1)$.
\end{theorem}

	By the Central Limit Theorem or Chernoff bound, it is easy to see that  with probability $1 - o(1)$, $|R| > |B|$ and the initial gap $\adv = (|R| - |B|)/2$ is $\omega(1/p)$.
	Conditioned on a fixed instance of $(R, B)$ satisfying this, we have $p\adv > 10$ and $p\adv^2 = \omega(1/p) = \omega(1)$, so we deduce by Theorem \ref{thm:main-nonlazy} that Red wins with probability $1 - O((p\adv^2)^{-1}) = 1 - o(1)$, proving the theorem.

\subsection{``Power of Few''  with random activations and updates}

In real-life situations, not everyone on a social network activates or updates their account on a daily basis. To make the model more realistic, we introduce random \emph{activation rate} $\pac$ and \emph{updating rate} $\pup$. 


%
%

\begin{definition}[Majority Dynamics with random activations and updates]
	Given a graph $G = (V, E)$, \emph{majority dynamics with activation rate $\pac$ and updating rate $\pup$} on $G$ is a process that works as follows. Beginning with Day 0, the vertices are colored either Red or Blue. For each $t\ge 0$, on day $t + 1$, the following takes place:
    \begin{itemize}
        \item Each vertex activates with probability $\pac$, independently.
        Neighboring vertices that are both active observe each other's color on day $t$.
        Inactive vertices neither observe nor are observed by anyone.

        \item Each active vertices then chooses to update its status with probability $\pup$, again independently.
        Inactive vertices do not update by default.

        \item Each updating vertex adopts the majority among the colors it observed or retains its day $t$ color in case of a tie.
    \end{itemize}
\end{definition}

It is clear that in this more general setting, unanimity is still permanent once it occurs.
We show that the ``Power of Few'' phenomenon remains robust here.


\begin{theorem}[The Power of Few with random activations and updates] \label{thm:main-lazy}
	Consider the majority dynamics process with fixed activation rate $\pac$ and updating rate $\pup$ on a $G(n, p)$ random graph, with $n/2 + \adv$ initial Red vertices.
	Assume $p = p(n)$ satisfies $p\pac n \ge (1 + \lambda)n^{-1}\log n$ for a constant $\lambda > 0$.
	Then there is a constant $C$ depending only on $\pac$ and $\pup$ such that when $p\adv \ge C$, a Red unanimity is reached within time $O_{\pac,\pup}(\log n)$ with probability at least
	\begin{equation*}
		1 - O\left(\frac{1}{\min\{n, p\adv^2\}}\right) - O(n^{-\lambda/4}) - n^{-\Omega_{\pac,\pup}(1)}.
	\end{equation*}
\end{theorem}

\begin{remark} \label{remark:main-thm-nonlazy-application}
	Despite seeming more general, Theorem \ref{thm:main-lazy} does not  directly imply \ref{thm:main-nonlazy}, due to different convergent times. 
	
  If we let $\pup = \pac = 1$ in Theorem \ref{thm:main-lazy}, then the setting is identical to Theorem \ref{thm:main-nonlazy}'s. However,  the former only implies convergence to unanimity in $O(\log n)$ days, while the latter gives a shorter time $O(\log_{pn}(n)) = o(\log n)$. 
  In particular, when  $pn \ge n^\alpha$ for some $\alpha \in (0, 1)$, Theorem \ref{thm:main-nonlazy} shows unanimity in $O(1/\alpha) = O(1)$ days, while the convergence time in Theorem \ref{thm:main-lazy} remains stuck at $O(\log n)$.
\end{remark}

\subsection{Open questions}

We describe some open questions and conjectures, for which we also present supporting evidence obtained through experiments.
The source codes to all experiments found in this section can be found in the Github repository \href{https://github.com/thbl2012/The_Power_of_Few_sparse_case.git}{\texttt{The\_Power\_of\_Few\_sparse\_case}}\cite{simulations2023}.

\subsubsection{``Optimal'' condition for unanimity}

The most interesting open question here is to determine the best value of ``Few'': 

\begin{question}  \label{qn:optimal-power-of-few}
	Given a density $p = p(n) \ge (1 + \lambda)n^{-1}\log n$, what is the optimal  initial gap $\adv = \adv(n)$ for which Red wins with probability arbitrarily close to $1$?
\end{question}
Let $\adv_{min}(n, p, \eps)$ be the smallest $\adv$ so that $\Pr(\text{Red wins}) \ge 1 - \eps$.
Theorem \ref{thm:main-nonlazy} gives an upper bound, namely $\max\{10p^{-1}, C_\eps p^{-1/2}\}$ for a sufficiently large constant $C_\eps$. When $p = o_n(1)$, this bound reduces to $10 p^{-1}$.  We conjecture that 
$p^{-1/2} $ is the right order of magnitude, as shown in the following experiment.
\begin{experiment}
[\href{https://github.com/thbl2012/The_Power_of_Few_sparse_case/blob/88f7b5d3b3dacfbdd29e12f0968c6bc7ce97f3e4/optimal_power_of_few.py}{\texttt{optimal\_power\_of\_few.py}} in \cite{simulations2023}]  \label{exp:optimal-power-of-few}
    We run the process with $n = 10^4$, $pn\in [9.3, 30]$, and $\Delta = fp^{-1/2}$ for $f\in \{1, 2, 3, 4, 5\}$.
For each pair $(p, f)$, we run $1000$ trials and record the winning percentage of Red.
The result shows that the winning rate stays mostly constant for each value of $p\adv^2$ as $p$ varies.
\end{experiment}
\begin{center}
    \includegraphics[height=2in]{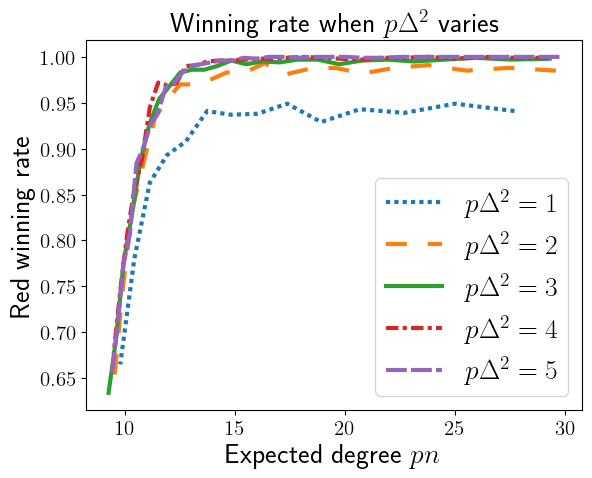}
\end{center}

Note that for $n = 10^4$, $\log n \approx 9.21$, so we only tested this phenomenon above the connectivity threshold.
The winning rate is already at least $.6$ at $pn = 9.3$, but it only stabilizes for $pn \approx 13$, as long as $p\Delta^2$ stays constant.
A further look into the data shows that in the case $9.3 < pn < 13$, in more than $90\%$ of the trials, Blue only retains at most 3 vertices, which is almost identical to a Red unanimity, barring one or two isolated Blue vertices or edges.

\begin{conjecture}[Optimal Power of Few]  \label{conj:optimal-power-of-few}
There is a function $f(x)$ such that 
$f(x) > 1/2$ for $x >0$, and $\lim_{x\to\infty} f(x) = 1$ such that the majority dynamics on $G(n, p)$ with initial advantage $\adv$ for Red and $p \ge (1 + \lambda)n^{-1}\log n$ ends unanimously Red with probability $f(\adv\sqrt{p}) - o_n(1)$. 
\end{conjecture}
 
 Sah and Sawhney \cite{sahsawhney2021}, improving a result by Devlin and Berkowitz  \cite{devlinberkowitz2022},
 proved the "Power of One" conjecture by the authors \cite[Conjecture 7]{TranVuDense2020}.
 A corollary of their main result
 confirms Conjecture \ref{conj:optimal-power-of-few} when $p$ tends to zero very slowly, namely
 $\log^{-1/16} n \le p $.
 The cases for smaller $p$ remain open.

This conjecture is related to a question concerning the random $1/2$ coloring scheme from the previous subsection:
{\it What is the smallest $p$ for which unanimity occurs with probability arbitrarily close to $1$?} 
The latest development concerning this question is a result 
by Chakraborti, Kim, Lee, and Tran \cite{jhkim2021}, which lowers the required $p$ from $\omega(n^{-1/2})$ in \cite{fountoulakis} to $\Omega(n^{-3/5}\log n)$.
Conjecture \ref{conj:optimal-power-of-few}, if true, gives the ultimate answer $p \ge (1 + \lambda)n^{-1}\log n$ for any constant $\lambda > 0$.

Incidentally, a conjecture by Benjamini et. al.  \cite{benjamini} predicts that for 
$p = \omega(n^{-1})$, then for any constant $\eps > 0$, one color will have at least $(1 - \eps)n$ vertices with probability $1 - o_n(1)$.
The condition $p = \omega(n^{-1})$, when combined with the fact that the random $1/2$ coloring scheme results in an initial difference of size $\Theta(n^{-1/2})$, essentially means the same as $\adv = \omega(p^{-1/2})$.

It is also natural to generalize the above conjectures with random activations and updates. It looks 
plausible that Conjecture \ref{conj:optimal-power-of-few}, if true, will also be robust in this more general model (in the spirit of Theorem \ref{thm:main-lazy}).



\subsubsection{The balanced initial coloring}

Another corollary of Sah's and Sawhney's \cite{sahsawhney2021} is that when $\Delta = 0$ and $p = \omega(\log^{-1/16}n)$, the side with the majority after the first day wins with probability $1 - o_n(1)$.
It is natural to ask what the lowest possible value of $p$ should be (in terms of $n$) for this to hold with high probability.

\begin{experiment}
[\href{https://github.com/thbl2012/The_Power_of_Few_sparse_case/blob/88f7b5d3b3dacfbdd29e12f0968c6bc7ce97f3e4/balanced_initial_colors.py}{\texttt{balanced\_initial\_colors.py}} in \cite{simulations2023}]
\label{exp:balanced-initial-coloring}
    We run the dynamic for $n = 10^4$, $pn \in [9, 19]$ (note that $\log n \approx 9.2$), and $\Delta = 0$ ($5000$ vertices on each side). Each trial stops when the process has entered a period of length $1$ or $2$ (whichever comes first).
    For each such $p$, we run $300$ trials, and record the stopping time and the sizes of the Red and Blue camps at that time.
\end{experiment}
\begin{center}
    \includegraphics[height=1.4in]{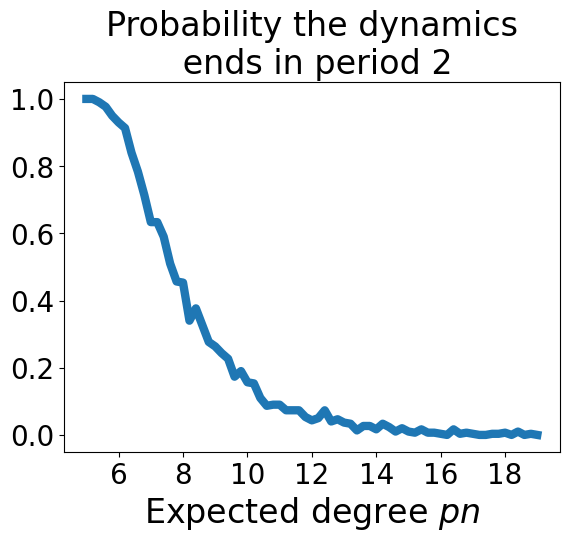}
    \quad
    \includegraphics[height=1.4in]{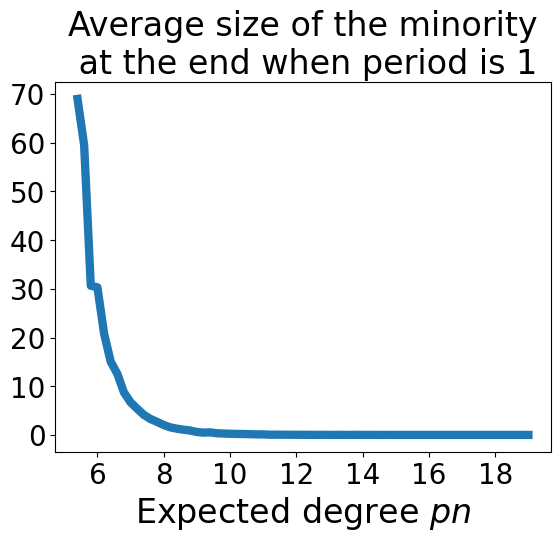}
    \quad
    \includegraphics[height=1.4in]{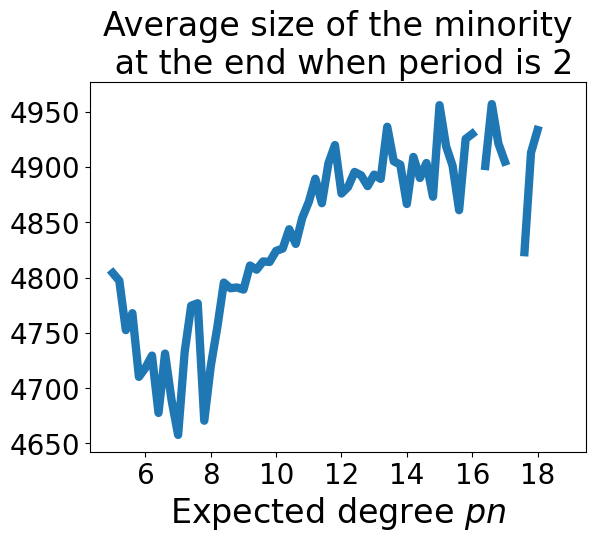}
\end{center}
The leftmost graph shows that the probability of ending up in period 2 declines continuously from near $1$ to near $0$ when $pn \ge 14$, meaning there are no abrupt transitions at the connectivity threshold.
The two other graphs reveal that the process's behavior changes drastically between the two possible periodicities.
When the period is $1$, one side either wins or takes almost every vertex except the isolated ones of the other color.
When the period is $2$, the coloring stays nearly balanced.
\begin{conjecture}  \label{conj:balanced-initial-coloring}
    Consider majority dynamics on a $G(n, p)$ random graph
    that begins with
    $n/2$ vertices for each color.
    Let $\mathcal{C}_1$ be the event that process ends in a stable coloring
    , namely one which agrees entirely with its previous coloring
    .
    Then when $pn = c\log n$, we have
    \[
    \Pr(\mathcal{C}_1) = f(c) - o_n(1),
    \]
    where $f$ is an increasing function such that $f(x) \xrightarrow{x\to 0^+} 0$, $f(x) \xrightarrow{x\to \infty} 1$ and $f(1) > .9$.
    Moreover, let $t^*$ be the smallest day such that $B_{t^*} = B_{t^* - 2}$, then for a universal constant $C$,
    \[
    \Pr\left(\min\{\card{R_{t^*}}, \card{B_{t^*}}\} \le Cne^{-pn}\mid \mathcal{C}_1\right) = 1 - o_n(1),
    \]
    and for any $\eps \in (0, 1)$,
    \[
    \Pr\left(\min\{\card{R_{t^*}}, \card{B_{t^*}}, \card{R_{t^* - 1}}, \card{B_{t^* - 1}}\} \ge (1/2 - \eps)n\mid \neg\mathcal{C}_1\right) = 1 - o_n(1).
    \]
\end{conjecture}
We believe that the period 2 is caused the emergence of a special structure in the graph, which also keeps both sides balanced with high probability, but we do not know what such a structure might be.

\subsubsection{The number of Blues in the giant component}

Another natural question is: how many Blue vertices still remain for $p < n^{-1}\log n$?
This number is at least the number of isolated Blue vertices, which is around $|B_0|e^{-pn}$.
However, there are also larger connected components that can retain a significant number of Blues.
\begin{experiment}
[\href{https://github.com/thbl2012/The_Power_of_Few_sparse_case/blob/88f7b5d3b3dacfbdd29e12f0968c6bc7ce97f3e4/colors_on_components.py}{\texttt{colors\_on\_components.py}} in \cite{simulations2023}]
\label{exp:blues-at-end}
    We run the dynamic for $n = 10^4$, $pn \in [1, 10]$ (since $\log n\approx 9.2$) with $|R_0| = 7000$, $|B_0| = 3000$, stopping whenever a period of $1$ or $2$ is reached.
    For each $p$, we run $500$ trials, and record the number of Blues in the beginning and the end in each connected component.
\end{experiment}
As observed in Experiment \ref{exp:balanced-initial-coloring}, the portion of non-isolated vertices in the Blue camp at the end go to $0$ in the regime $p = cn^{-1}\log n$, thus we restrict our attention to the case $1 < d = pn = O(1)$.
In this regime, a unique giant component
of order
$n(1 - O(e^{-d}))$
emerges with high probability.
Denote this component by $G^*$, the graphs below show the results in $G^*$, where $t^*$ is the ``ending day'' defined in Conjecture \ref{conj:balanced-initial-coloring}.
\begin{center}
    \includegraphics[height=1.8in]{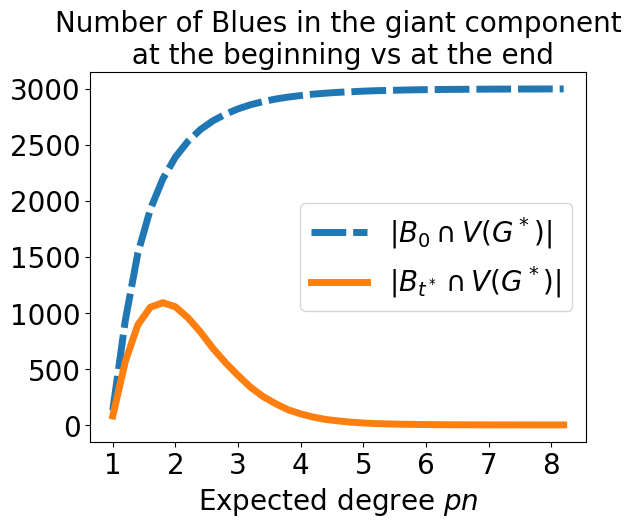}
    \quad
    \includegraphics[height=1.8in]{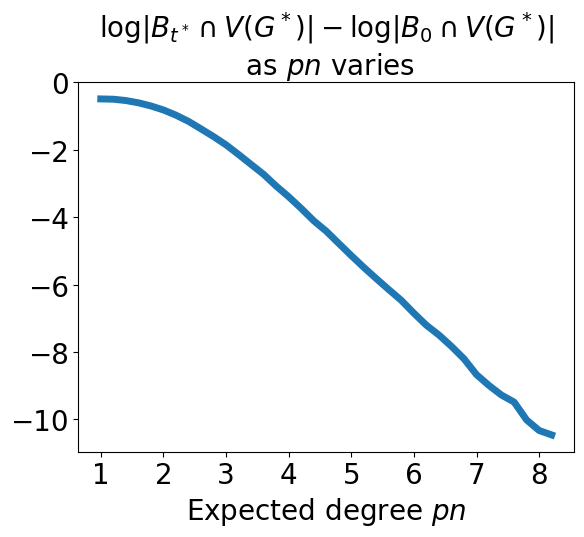}
\end{center}
The plot of the log ratio of $\card{B_{t^*} \cap V(G^*)}$ over $\card{B_0 \cap V(G^*)}$ shows an almost linear relationship.
We conjecture that this is the case.
\begin{conjecture}  \label{conj:blue_end_giant_log_ratio}
    Consider majority dynamics on a $G(n, p)$ random graph with $\eps n \le \card{B_0} \le (1/2 - \eps)n$ for some constant $\eps \in (0, 1/2)$.
    When $pn = d$ for a constant $d > 1$, there is a constant $a > 1$, and a polynomial $P$ where $P(x) > 0$ for $x > 0$, both independent of $\eps$ and $d$, such that
    \[
    \card{B_{t^*} \cap V(G^*)} \le \card{B_0 \cap V(G^*)}P(pn)e^{-apn},
    \]
    with probability $1 - o_n(1)$, where $t^*$ is the smallest day such that $B_{t^*} = B_{t^* - 2}$.
\end{conjecture}
We expect $a$ to be greater than 1 because the number of Blues in the giant component will be dominated by the number of isolated Blues when $pn$ is large.

In the rest of the paper, we present the proofs  of Theorems \ref{thm:main-nonlazy} and \ref{thm:main-lazy}.
In Section \ref{sec:p<log(n)/n}, we state and prove some results about majority dynamics using our technique when $p$ is below the connectivity threshold.

\section{Proof of Theorem \ref{thm:main-nonlazy}} \label{sec:main-nonlazy-proof}

\subsection{Overview}

The proof uses several notations, which we describe below, and several technical lemmas, whose proofs are put in
Appendices \ref{sec:prelim-lemmas} and \ref{sec:shrink-scheme}.

\begin{itemize}[topsep=0pt]
	\item $v\adj u$ denotes the event that $u$ and $v$ are neighbors.

    \item $\Neigh(u)$ denotes the set of neighbors of the vertex $u$.
	
	
	\item $\ca{N}(\mu, \sigma^2)$ denotes the Normal distribution with mean $\mu$ and variance $\sigma^2$.
	
	
	\item $\Phi(a) \defeq (2\pi)^{-1/2}\displaystyle\int_{-\infty}^a e^{-x^2/2}dx$ and $\Phi_0(a) \defeq \Phi(a) - 1/2$.
	
	\item $\Bin(n, p)$ denotes the binomial distribution of size $n$ and probability $p$.
 
	\item $\Ber(p)$ denotes the Bernoulli distribution with probability $p$.
	
	

    \item $\Cbe = .56$ is the constant in the latest version of the Berry-Esseen Theorem (Theorem \ref{thm:esseen}).

    \item When three terms $X, Y, E$ satisfy $|X - Y|\le E$, we write $X = Y \pm E$.
\end{itemize}

Additionally, we use $\Clr$ to denote a coloring, namely a function from $V$ to $\{-1, 1\}$, so that the Red and Blue camps are respectively given the notations $R \defeq \Clr^{-1}(1)$ and $B \defeq \Clr^{-1}(-1)$.
Given $\Clr$, we also define a function $\dif: V\to \Z$ that sends each $v$ to the numerical difference between its Red neighbors and Blue neighbors:
\begin{equation} \label{eq:dif-formula}
    \dif(v) \defeq \card{R\cap \Neigh(v)} - \card{B\cap \Neigh(v)} = \sum_{u\in V} \Clr(u)\mathbf{1}_{u\adj v},
\end{equation}
To denote the coloring on day $t$, we use $\Clr_t$, and accordingly, $R_t$, $B_t$ and $\dif_t$.

The proof of Theorem \ref{thm:main-nonlazy} has three steps, spanning the next 3 subsections.
\begin{enumerate}
	\item With probability at least $1 - O\left(\min\{n, p\adv^2\}^{-1}\right)$, the number of Blue nodes shrinks below $n/2 - \Omega\left(\min\left\{n, \adv\sqrt{pn}\right\}\right)$ after the first day.
	
	\item With probability at least $1 - e^{-\Omega(n)}$, Blue shrinks below $n/2 - \Omega(n)$ after the second day.
	
	\item With probability at least $1 - O(n^{-\lambda/2})$, $G$ has ``nice'' properties that allow Blue to shrink repeatedly until no Blue nodes are left.
\end{enumerate}
We begin with Step 1 in the next section.

\subsection{Step 1 - First day} \label{sec:day1}

Below is the main technical result of this section.

\begin{lemma} \label{lem:day1}
	For every $0 < c \le \min\bigl\{p\adv, \sqrt{p(1 - p)n}\bigr\}$ and $D\in \mathbb{R}$ define
	\begin{equation} \label{eq:temp-D0}
		D_c \defeq \Phi_0\left(2 - \frac{1}{c}\right) - \frac{\Cbe}{c}
		\quad \text{ and } \quad
		C_{\ref{lem:day1}}(c, D) \defeq \frac{7}{12\left(D_c - D \right)^2}.
	\end{equation}
	For every $0 < D < D_c$, if $\card{B_0} \le n/2 - \adv$ then
	\[
	\Pr\left(\card{B_1} \le \frac{n}{2} - D\min\{n, \adv\sqrt{pn}\}\right) \ge 1 - \frac{C_{\ref{lem:day1}}(c, D)}{\min\{n, p\adv^2\}}.
	\]
\end{lemma}



As mentioned in the overview, we can prove Lemma \ref{lem:day1} by combining Chebyshev's inequality with the following proposition.
\begin{proposition} \label{prop:day1-exp-var}
	For each $0 < c \le \min\{p\adv, \sqrt{p(1 - p)n}\}$ we have
	\[
	\E{\card{B_1}} \le n/2 - D_c\min\{n, \adv\sqrt{pn}\}
	\quad \text{ and } \quad
	\Var{\card{B_1}} \le 7n/12.
	\]
\end{proposition}

Assume, for now, Proposition \ref{prop:day1-exp-var}.
Let $m \defeq \min\{n, \adv\sqrt{pn}\}$.
For any $D < D_c$,
Chebyshev's inequality implies
\begin{equation}
	\begin{aligned}
		& \Pr\left(\card{B_1} \ge \frac{n}{2} - Dm\right)
		\le \Pr\left(\card{B_1} \ge \E{\card{B_1}} + (D_c - D)m\right)
		\le \frac{\Var{\card{B_1}}}{(D_c - D)^2m^2}.
	\end{aligned}
\end{equation}
The right-most expression can be simplified to $\frac{C_{\ref{lem:day1}}(c, D)}{\min\{n, p\adv^2\}}$, giving Lemma \ref{lem:day1}.

\begin{proof}[Proof of Proposition \ref{prop:day1-exp-var}]
	Write $\card{B_1} = \sum_{v\in V}\one_{B_1}(v)$.
    We will bound $\card{B_1}$ using Chebyshev's inequality.
    We proceed by bounding $\E{\card{B_1}}$ and $\Var{\card{B_1}}$.
    \emph{Step 1. Bounding the expectation.}
    We have
	\begin{equation} \label{eq:blue1-exp}
		\E{\card{B_1}} = \sum_{v\in V}\E{\one_{B_1}(v)} \le \sum_{v\in V}\Pr(\dif_0(v) \le 0).
	\end{equation}
	Recall that by Eq. \ref{eq:dif-formula},
	$
	\dif_0(v) = \sum_{u\in V} \Clr_0(v)\mathbf{1}_{u\adj v}
	$,
	where $\{\mathbf{1}_{u\adj v}\}_{u\in V}$ are i.i.d. $\Ber(p)$ variables.
	By Gaussian approximation (Corollary \ref{lem:strongesseen} in Appendix \ref{sec:prelim-lemmas}), we have
	\[
	\Pr(\dif_0(v) \le 0)
	\le \frac{1}{2} - \Phi_0\left(\frac{\E{\dif_0(v)}}{\sqrt{\Var{\dif_0(v)}}}\right)
	+ \frac{\Cbe}{\sqrt{\Var{\dif_0(v)}}}.
	\]
	Calculations show
	$\E{\dif_0(v)} \ge p(2\adv - 1)$ and $\Var{\dif_0(v)} = p(1 - p)(n - 1)$.
	Using the inequality $\Phi_0(x) \ge \Phi_0(y)\min\{x/y, 1\}$, we have for arbitrary $y > 0$:
	\[
	\begin{aligned}
		\E{\one_{B_1}(v)}
		& \le \frac{1}{2} - \min \left\{1, \frac{p(2\adv - 1)}{y\sqrt{p(1 - p)n}}\right\} \Phi_0(y) + \frac{\Cbe}{\sqrt{p(1 - p)n}} \\
		& = \frac{1}{2} - \min \left\{\Phi_0(y) - \frac{\Cbe}{\sqrt{p(1 - p)n}}, \ \frac{2\frac{\Phi_0(y)}{y}p\adv - \bigl(p\frac{\Phi_0(y)}{y} + \Cbe\bigr)}{\sqrt{p(1 - p)n}} \right\}.
	\end{aligned}
	\]
	Using the fact
    $c\le \min\{p\adv, \sqrt{p(1 - p)n}\}$,
    the RHS can be bounded further:
    \begin{equation*}
		\Phi_0(y) - \frac{\Cbe}{\sqrt{p(1 - p)n}} \ge \Phi_0(y) - \frac{\Cbe}{c},
    \end{equation*}
    and
    \begin{equation*}
		\frac{2\frac{\Phi_0(y)}{y}p\adv - \bigl(p\frac{\Phi_0(y)}{y} + \Cbe\bigr)}{\sqrt{p(1 - p)n}}
        \ge \left(2\frac{\Phi_0(y)}{y} - \frac{\Phi_0(y)/y + \Cbe}{c}\right) \frac{p\adv}{\sqrt{p(1 - p)n}}
	\end{equation*}
	Now let $y = 2 - c^{-1}$ so that the RHSs of the above two equations become $D_c$ and $D_c \frac{p\adv}{\sqrt{p(1 - p)n}}$ respectively.
	We have
	\[
	\E{\one_{B_1}(v)} \le \frac{1}{2} - D_c\min \left\{1, \frac{p\adv}{\sqrt{p(1 - p)n}} \right\}
	\le \frac{1}{2} - D_c\min \left\{1, \frac{\Delta\sqrt{p}}{\sqrt{n}} \right\}
	\]
	Using Eq. \ref{eq:blue1-exp}, we get the desired bound on $\E{\card{B_1}}$.

%
%

    \emph{Step 2. Upper-bounding the variance.}
	Firstly write
	\[
	\Var{\card{B_1}} = \sum_{v} \Var{\one_{B_1}(v)} + 2\sum_{u < v} \Cov{\one_{B_1}(u), \one_{B_1}(v)}.
	\]
	Now that $\Var{\one_{B_1}(v)} = \E{\one_{B_1}(v)}(1 - \E{\one_{B_1}(v)}) \le 1/4$ for each $v$, we are left to bound $\Cov{\one_{B_1}(u), \one_{B_1}(v)}$ for each pair $u < v$.
	Let us write
	\begin{equation} \label{eq:day1-cov-1}
		\Cov{\one_{B_1}(u), \one_{B_1}(v)} = \Pr(u, v\in B_1) - \Pr(u\in B_1)\Pr(v\in B_1).
	\end{equation}
	We can break up these three events into independent ones as follows.
	\begin{equation} \label{eq:day1-cov-2}
		\Pr\bigl(u, v\in B_1\bigr)
		= p\Pr\bigl(u, v\in B_1\mid u\adj v\bigr) + (1 - p)\Pr\bigl(u, v\in B_1\mid u\not\adj v\bigr).
	\end{equation}
	Let $F_u \defeq \bigl\{u\in B_1 \big| u\adj v\bigr\}$ and
    similarly
    for $F_v$.
	$F_u$ and $F_v$ are independent since $u$ and $v$ have no influence over the other's neighbors in $V\setminus\{u, v\}$.
	
 The second component in the RHS of Eq. \ref{eq:day1-cov-2} is thus
	$
	\Pr(F_u)\Pr(F_v)
	$.
	By the same reasoning, let $E_u \defeq \bigl\{u\in B_1 \big| u\not\adj v\bigr\}$ and
    similarly
    for $E_v$, then $E_u$ and $E_v$ are independent and the first component in the RHS of Eq. \ref{eq:day1-cov-2} is $\Pr(E_u)\Pr(E_v)$.
	
	The probabilities $\Pr(u\in B_1)$ and $\Pr(v\in B_1)$ can be broken down easily.
	We have $\Pr(u\in B_1) = p\Pr(E_u) + (1 - p)\Pr(F_u)$ and
    analogously
    for $\Pr(v\in B_1)$.
	Substituting these back into \ref{eq:day1-cov-2} and into \ref{eq:day1-cov-1} gives, after some algebraic manipulations:
	\[
	\Cov{\one_{B_1}(u), \one_{B_1}(v)} = p\left(1 - p\right) \left(\Pr(E_u) - \Pr(F_u)\right) \left(\Pr(E_v) - \Pr(F_v)\right).
	\]
	The term $\Pr(E_v) - \Pr(F_v)$ can be handled by looking at their raw meanings.
	Letting
	\[
	X = \dif_0(v) - \Clr_0(u)W_{uv} + \one_{B_0}(v),
	\]
	we see that $E_v \iff X \le 0$ and $F_v \iff X \ge -\Clr_0(u)$, so
	\[
	\begin{aligned}
		\Pr(E_v) - \Pr(F_v) = \Pr(X \le 0) - \Pr(X + \Clr_0(u) \le 0) = \Clr_0(u)\Pr\left(X = \frac{\Clr_0(u) - 1}{2}\right),
	\end{aligned}
	\]
	where the last equality can be reached by case analysis on $u$'s color.
	By Lemma \ref{lem:esseen-err-term} in Appendix \ref{sec:prelim-lemmas},
	\[
	\Pr(X = \Clr_0(u) - 1) \le \frac{\Cbe}{\sqrt{p(1 - p)(n-2)}}.
	\]
	We get the same inequality for $u$, therefore
	\[
	\Cov{\one_{B_1}(u), \one_{B_1}(v)} \le \frac{\Cbe^2}{n - 2} < \frac{1}{3n}.
	\]
	
	\noindent Finally we get:
	$ \displaystyle \
	\Var{\card{B_1}} < \sum_{v} \frac{1}{4} + \sum_{u < v} \frac{1}{3n} < \frac{n}{4} + \frac{n}{3} = \frac{7n}{12}.
	$
\end{proof}

The proof of Lemma \ref{lem:day1} is now complete and we consider Step 1 done.


\subsection{Step 2 - Second day} \label{sec:day2}

We first introduce one of the two main technical lemmas of this step (the other one being Lemma \ref{lem:hypo-blue-bound-2}).
It shows that if the gap is at least $C\sqrt{n}$ for a sufficiently large constant $C$, Blue will shrink with a factor less than one on the next day, regardless of which vertices are Blue today.

\begin{lemma} \label{lem:hypo-blue-bound}
	For fixed $x, y\in (0, 1/2)$ and $d > 0$, define
	\begin{align}
		h(x) & \defeq x\log x + (1 - x)\log (1-x), \label{eq:h-defn} \\
		F(d, x, y) & \defeq dx\left(1 - \frac{1}{2}y - \frac{3}{2}y^{1/3}(1 - y)^{2/3}\right) + 2h(x)
		\label{eq:F-defn}.
	\end{align}
	Then for any $0\le b_2\le b_1 \le 1/2 - 7/n$,
	$G$ is such that for any coloring with at most $b_1n$ Blue nodes, there are at most $b_2n$  nodes that can turn Blue the next day,
	except with probability at most $b_2^{-1}\exp\bigl[-nF(pn, b_2, b_1)\bigr]$.
\end{lemma}



\begin{proof}
    Fix two subsets $B \in \binom{V}{\lfloor xn \rfloor}$ and $S \in \binom{V}{\lfloor yn \rfloor + 1}$.
    Since the Blue camp on the next day is fully determined from the current Blue camp and the graph $G$, we can let $\mathcal{E}(B, S)$ be the event that $S$ is contained in the Blue camp on the next day if $B$ is the Blue camp today.
    Note that this is an event on the space of all simple, undirected graphs over $n$ vertices, endowed with the $G(n, p)$ measure.
    Let $\dif(S) \defeq \sum_{v\in S} \dif(v)$, which is a random variable on the same space, and notice that $\mathcal{E}(B, S)$ implies $\dif(S)\le 0$.
	We apply a Chernoff argument: for any $t > 0$,
    \[
	\Pr(\mathcal{E}(B, S)) \le \Pr(\dif(S) \le 0) \le \Pr(e^{-t\dif(S)} \ge 1) \le \E{e^{-t\dif(S)}}.
	\]
	To proceed, we need to split $\dif(S)$ into independent random variables.
	\[
	\begin{aligned}
		\dif(S) & = \sum_{v\in V} \Bigl[\chi_S(v)\sum_{u\in V}\Clr(u)\mathbf{1}_{u\adj v}\Bigr]
		= \sum_{\{u, v\}\subset V} \Bigl[\Clr(u)\chi_S(v) + \Clr(v)\chi_S(u) \Bigr] \mathbf{1}_{u\adj v} \\
		& = \sum_{\{u, v\}\subset V} X(u, v)\mathbf{1}_{u\adj v},
		\quad \quad \text{ for } X(u, v) \defeq \Clr(u)\chi_S(v) + \Clr(v)\chi_S(u).
	\end{aligned}
	\]
	We have
	\begin{equation} \label{eq:difS-bound-1}
		\begin{aligned}
			& \log\E{e^{-t\dif(S)}} = \sum_{\{u, v\}\subset V} \log\E{e^{-tX(u, v)\mathbf{1}_{u\adj v}}} \\
			& = \sum_{\{u, v\}\subset V} \log\Bigl(1 - p + p\E{e^{-tX(u, v)}} \Bigr)
			\le p \sum_{\{u, v\}\subset V} \Bigl(\E{e^{-tX(u, v)}} - 1 \Bigr).
		\end{aligned}
	\end{equation}
	Table \ref{tab:difS-summands} sums up respective values for $X(u, v)$ when $u$ and $v$ vary among the four sets:
	$R\cap S$, $R\cap S^c$, $B\cap S$, $B\cap S^c$ (where $R = B^c$ is the set of Red nodes).
	\begin{table}[h]
		\[\begin{array}{c|c|c|c|c}
			u\setminus v & R\cap S               & B\cap S               & R\cap S^c & B\cap S^c \\ \hline
			R\cap S      & 2 & 0 & \multirow{2}{*}{$1$} & \multirow{2}{*}{$-1$} \\ \cline{1-3}
			B\cap S      & 0 & -2 &           &  \\ \hline
			R\cap S^c    & \multicolumn{2}{|c|}{1}               & \multicolumn{2}{|c}{\multirow{2}{*}{0}}  \\ \cline{1-3}
			B\cap S^c    & \multicolumn{2}{|c|}{-1}              & \multicolumn{2}{|c}{}
		\end{array}\]
		\caption{Values of $X(u, v)$}
		\label{tab:difS-summands}
	\end{table}
	Using this table, we get
	\begin{equation} \label{eq:difS-bound-2}
		\begin{aligned}
			\sum_{\{u, v\}\subset V} \Bigl(\E{e^{-tX(u, v)}} - 1 \Bigr)
			= & \binom{\card{S\cap R}}{2}\left(e^{-2t} - 1\right) + \binom{\card{S\cap B}}{2}\left(e^{2t} - 1\right)\\
			& + \card{S}\card{R\cap S^c}(e^{-t} - 1) + \card{S}\card{B\cap S^c}(e^t - 1)
		\end{aligned}
	\end{equation}
	We aim to find a simple upper bound for this expression with an appropriate $t$.
	It can be viewed as $(1 - e^{-t})f(\card{S\cap R})$, where
	\begin{equation*}
		f(x) =  \binom{\card{S} - x}{2}e^t\left(e^{t} + 1\right)
		-\binom{x}{2}\left(e^{-t} + 1\right)
		- \card{S}(\card{R} - x) + \card{S}(\card{B} - \card{S} + x)e^t,
	\end{equation*}
	is a quadratic polynomial of $x$, with quadratic coefficient
	$\frac{1}{2}(e^t - e^{-t})(e^t + 1) > 0$.
	Thus $f$ is convex and $f(x) \le \max\{f(0), f(\card{S})\}$ for $x\in [0, \card{S}]$.
	Let
	\[
	T = (1/2)\card{S}n\left(x\left(e^{2t} + e^t + 2\right) - 2\right),
	\]
	we prove $\max\{f(0), f(\card{S})\} \le T$ for an appropriately chosen $t$.
	Replacing $\card{B}$ with $\lfloor xn \rfloor$ and $\card{S}$ with $\lfloor yn \rfloor + 1$, we have after a few computations,
	\[
	T - f(0) = \card{S}\Bigr(e^t + (xn - \lfloor xn \rfloor) + (xn - \lfloor yn \rfloor)e^t(e^t - 1)/2 \Bigr) > 0
	\ \text{ for all } t > 0,
	\]
	On the other hand,
	\[
	\begin{aligned}
		T - f(\card{S})
		& = \card{S}\Bigl((xne^t - \lfloor yn \rfloor e^{-t})(e^t - 1)/2 + (xn - \lfloor xn \rfloor)(e^t + 1) - 1 \Bigr) \\
		& \ge \frac{1}{2}\card{S}(xn(e^t - 1)^2(1 + e^{-t}) - 2) \ge \frac{1}{2}\card{S}(xn(e^t - 1)^2 - 2),
	\end{aligned}
	\]
	which is not always positive. Let us call a $t$ ``nice'' if $xn(e^t - 1)^2 \ge 2$.
	Assuming we have a nice $t$, then by Eqs. \eqref{eq:difS-bound-1} and \eqref{eq:difS-bound-2},
	\[
	\log\E{e^{-t\dif(S)}} \le pn\card{S}g_x(t), \ \text{ where } g_x(t) \defeq (1 - e^{-t})\left(x(e^{2t} + e^t + 2)/2 - 1\right)
	\]
	If we were to choose $t$ freely, the best choice would be $t_x = \log(1/x - 1)/3$, which can be shown easily to minimize $g_x(t)$.
    Fortunately, we can show that this $t_x$ is nice.
    Using the fact
    \begin{equation*}
        (y - 1)(y^2 + y + 1) = y^3 - 1 \quad \text{ for all } y,
    \end{equation*}
    specifically at $y = e^{t_x}$, we have
    \begin{equation*}
    \begin{aligned}
        xn(e^{t_x} - 1)^2 \ge 2
        & \iff xn(e^{3t_x} - 1)^2 \ge 2(e^{2t_x} + e^{t_x} + 1)^2
        \\
        & \iff n(1 - 2x)^2 \ge 2x(e^{2t_x} + e^{t_x} + 1)^2,
    \end{aligned}
    \end{equation*}
    where the last inequality is obtained by plugging the definition of $t_x$ into the left-hand side.
    Now consider two cases:
	If $x < 1/9$, we have $e^{2t_x} + e^{t_x} + 1\le 3e^{2t_x} < 3x^{-2/3}$,
	and $1 - 2x > 7/9$, so it suffices to show that
	$n(7/9)^2 \ge 18x^{-1/3}$, or $xn^3 \ge (1458/49)^3$, which holds when $x \ge 1/n$ and $n$ is large enough.
	
	If $x \ge 1/9$, $e^{t_x} \le 2$, and $2x \le 1$ so it suffices to show that
	$n(1 - 2x)^2 \ge 7$.
	By hypothesis, $x\le 1/2 - 2n^{-1/2}$ so $1 - 2x \ge 4n^{-1/2}$, thus $n(1 - 2x)^2 \ge 16 > 7$.
	
	Therefore we can use $t_x = \log(1/x - 1)/3$ and obtain the bound
	\begin{equation}  \label{eq:difS-bound-3}
		\begin{aligned}
			\Pr(\mathcal{E}(B, S))
			& \le \E{e^{-t\dif(S)}} \le e^{-pn\card{S}g(t_x)} \\
			& \le \exp\left[-pn^2y\left(1 - \frac{1}{2}x - \frac{3}{2}x^{1/3}(1 - x)^{2/3}\right)\right].
		\end{aligned}
	\end{equation}
	Let $\mathcal{G}$ be the event that for any choice of at most $xn$ Blue nodes, there are at most $yn$ Blue nodes the next day.
	$\mathcal{G}^c$ is the event that there are $B\in \binom{V}{\lfloor xn \rfloor}$ and $S\in \binom{V}{\lfloor yn \rfloor + 1}$ satisfying $\mathcal{E}(B, S)$.
	By a double union bound, $\Pr(\mathcal{G}^c) \le \binom{n}{\lfloor xn \rfloor}\binom{n}{\lfloor yn \rfloor + 1}\Pr(\mathcal{E}(B, S))$.
	A routine Stirling bound shows that with $h(x) = x\log x + (1 - x)\log(1 - x)$, $\binom{n}{k} \le e^{-nh(k/n)}$.
	Therefore we have
	\begin{equation}  \label{eq:difS-bound-4}
		\begin{aligned}
			& \binom{n}{\lfloor xn \rfloor}\binom{n}{\lfloor yn \rfloor + 1}
			\le \exp\left[-n\left(h\Bigl(\frac{\lfloor xn \rfloor}{n}\Bigr) + h\Bigl(\frac{\lfloor yn \rfloor + 1}{n}\Bigr)\right)\right] \\
			& \le \exp\left[-n\left(h(x) + h\Bigl(y + \frac{1}{n}\Bigr)\right)\right].
			\quad (h \text{ is decreasing on } \left(0, \frac{1}{2}\right))
		\end{aligned}
	\end{equation}
	Since $h$ is convex, $h(y + \eps) \ge h(y) + \eps h'(y)$, so
	\begin{equation}  \label{eq:difS-bound-5}
		\begin{aligned}
			\exp\left[-n\left(h(x) + h\Bigl(y + \frac{1}{n}\Bigr)\right)\right]
			& \le \exp\left[-n\left(h(x) + h(y) + \frac{1}{n}\log\left(\frac{y}{1 - y}\right)\right)\right] \\
			& \le \frac{1 - y}{y}e^{-n(h(x) + h(y))} \le \frac{1}{y}e^{-2nh(x)}.
		\end{aligned}
	\end{equation}
	Combining Eqs. \eqref{eq:difS-bound-3}, \eqref{eq:difS-bound-4} and \eqref{eq:difS-bound-5}, we have
	\[
	\Pr(\mathcal{G}^c) \le \frac{1}{y}\exp\left[-pn^2y\left(1 - \frac{1}{2}x - \frac{3}{2}x^{1/3}(1 - x)^{2/3}\right) - 2nh(x)\right].
	\]
	The right-hand side is precisely $(1/y)\exp\bigl[-nF(pn, y, x)\bigr]$, as desired.
\end{proof}

Assuming Lemma \ref{lem:hypo-blue-bound}, the following holds.

\begin{lemma} \label{lem:step2}
	For any $a > 0$ and $b\in (0, 1/2)$ satisfying:
	\begin{equation} \label{eq:step2-abc-cond}
		ba^2 > 3(\log 2)/2,
	\end{equation}
	if $\card{B_1} \le \frac{n}{2} - \frac{an}{\sqrt{pn}}$,
	then $\card{B_2} \le bn$ with probability at least $1 - e^{-\Omega(n)}$.
\end{lemma}

\begin{proof}
	Let $b_1 = \frac{1}{2} - \frac{a}{\sqrt{pn}}$.
	Conditioned on $\card{B_1} \le b_1n$, Lemma \ref{lem:hypo-blue-bound} implies
	\[
	\Pr\left(\card{B_2} \le bn\right) \ge 1 - b^{-1}\exp\bigl[-nF(pn, b, b_1)\bigr].
	\]
	We want to show $F(pn, b, b_1) = \Omega(1)$, where by Eq. \ref{eq:F-defn}, we have
	\[
	F(pn, b, b_1) = pnb\left(1 - \frac{1}{2}b_1 - \frac{3}{2}b_1^{1/3}(1 - b_1)^{2/3}\right) + 2h(b_1).
	\]
	A routine Taylor expansion around $1/2$ shows that
	\begin{equation} \label{eq:step2-temp1}
		1 - \frac{1}{2}b_1 - \frac{3}{2}b_1^{1/3}(1 - b_1)^{2/3} \ge \frac{1}{3}(1 - 2b_1)^2 = \frac{4a^2}{3pn},
		\ \text{ and } \
        h(b_1) \ge -\log 2.
	\end{equation}
	Thus
	$F(pn, b, b_1) \ge \frac{4}{3}a^2b - 2\log 2 = \Omega(1)$ due to \eqref{eq:step2-abc-cond}.
	The proof is complete.
\end{proof}

\subsection{Step 3 - Third Day onwards} \label{sec:day3+}

Note that although Lemma \ref{lem:hypo-blue-bound} can be used to prove Blue repeatedly shrinks until it vanishes,
some calculations show that the minimum $p$ for which this argument works is $2n^{-1}\log n$.
We use a different argument in this section specifically to include the case $p = (1 + \lambda)n^{-1}\log n$, which is also the smallest possible threshold.

\begin{lemma} \label{lem:hypo-blue-bound-2}
	Let $A_G$ denote the adjacency matrix of $G$.
	For any choice of $B_t$ and any $1/2 > b \ge \card{B_t}/n$, we have either $B_{t+1} = \varnothing$ or
	\[
	\card{B_{t+1}} \Bigl(\frac{1}{\card{B_{t+1}}}\sum_{v\in B_{t+1}}d(v) \Bigr)^2 \le \inf_x\|A_G - x\mathbf{1}_n\mathbf{1}_n^T\|_{op}^2 \frac{1 - b}{(1/2 - b)^2}\card{B_t},
	\]
    where $\|\cdot\|_{op}$ denotes the matrix operator norm.
\end{lemma}

This lemma works best when $\card{B_t} \le bn$ for a constant $b < 1/2$, hence the need for the argument in Step 2.
The full proof is presented in Appendix \ref{sec:shrink-scheme}.

\begin{lemma} \label{lem:step3}
	Consider the majority dynamics process on $G\sim G(n, p)$ where $pn \ge (1 + \lambda)\log n$.
	Assume that $\card{B_2} \le bn$ for some constant $b < 1/2$.
	Then with probability at least $1 - 2n^{-\lambda/2}$, there is a time $t^* = O(\log_{pn}n)$ such that $\card{B_{t^*}} = \varnothing$.
\end{lemma}

\begin{proof}
	Lemma \ref{lem:matrix-norm-bound} implies that with probability at least $1 - n^{-\lambda}$,
	\begin{equation}  \label{eq:step3-temp1}
	\|A_G - p\mathbf{1}_n\mathbf{1}_n^T\|_{op}
	\le \|A_G - \E{A_G}\|_{op} + \|pI_n\|_{op} < (C + 1)\sqrt{pn},
	\end{equation}
    for the constant $C = C_{\lambda + 1, \lambda}$ from Lemma \ref{lem:matrix-norm-bound}.
	Lemma \ref{lem:degree-lower-bound} implies that for $\alpha = \alpha_\lambda$ in Lemma \ref{lem:degree-lower-bound},  with probabiliy at least $1 - n^{-\lambda/2}$,
	\begin{equation}  \label{eq:step3-temp2}
	\min{\{d(v): v\in V\}} \ge \alpha pn.
	\end{equation}
	Thus with probability $1 - 2n^{-\lambda/2}$, we can assume $G$ satisfies Eqs. \eqref{eq:step3-temp1} and \eqref{eq:step3-temp2}.
	The latter also implies the average degree in $B_{t+1}$ is at least $\alpha pn$.
	Assuming $\card{B_t}\le bn$, Lemma \ref{lem:hypo-blue-bound-2} implies that
	\[
	\card{B_{t+1}} (\alpha pn)^2 \le (C + 1)^2 pn \cdot \frac{1 - b}{(1/2 - b)^2} \card{B_t},
	\]
	or, equivalently,
	\[
	\card{B_{t+1}} \le \frac{A}{pn} \card{B_t}
	\quad \text{ where } A = \frac{(C + 1)^2(1 - b)}{\alpha^2(1/2 - b)^2}.
	\]
	By choosing $n$ large enough so that $pn \ge \log n > A$, we also have $\card{B_{t+1}} < bn$.
	This ensures the inequality above holds for $t+1$ and so on.
	After $\lceil \log_{pn/A}{bn} \rceil = O(\log_{pn}n)$ days, Blue has less than $1$ vertex and thus is empty.
\end{proof}

\subsection{Wrapping up the proof of Theorem \ref{thm:main-nonlazy}}

Consider the majority dynamics process on $G \sim G(n, p)$, where $1 - \frac{11}{n} \ge p \ge (1 + \lambda)n^{-1}\log n$ and $\card{R_0} = n/2 + \adv$ where $p\adv \ge 10$.
For $n$ sufficiently large, $\min\{\sqrt{p(1 - p)n}, p\adv\} \ge 10$.
Applying Lemma \ref{lem:day1} with $c = 10$ and $D = .21 < D_c \approx .41$, we have
\begin{equation} \label{eq:main-nonlazy-proof-step1}
	\card{B_1} \le \frac{n}{2} - D\frac{\min\{\sqrt{pn}, p\adv\}}{\sqrt{pn}}n,
\end{equation}
except with a fail probability at most $\frac{7}{12(D_c - D)^2}\min\{n, p\adv^2\}^{-1} < 14\min\{n, p\adv^2\}^{-1}$.
Conditioning on Eq. \ref{eq:main-nonlazy-proof-step1}, we apply Lemma \ref{lem:step2} with $a = D\min\{\sqrt{pn}, p\adv\} \ge 2.2$ and $b = .25$, which satisfy Eq. \ref{eq:step2-abc-cond}, to get
\begin{equation} \label{eq:main-nonlazy-proof-step2}
	\card{B_2} \le .25n,
\end{equation}
except with a fail probability at most $e^{-\Omega(n)}$.
Conditioning on Eq. \ref{eq:main-nonlazy-proof-step2}, we apply Lemma \ref{lem:step3} with $b = .25$ to finish the proof.
The total fail probability is at most
\[
\frac{14}{\min\{n, p\adv^2\}} + e^{-\Omega(n)} + 2n^{-\lambda/2}
= O\left(\frac{1}{\min\{n, p\adv^2\}} + n^{-\lambda/2}\right),
\]
which is the desired bound. \qed

\section{Some Results with Sub-connectivity Density} \label{sec:p<log(n)/n}

In this section, we state and prove some results about Majority Dynamics on a $G(n, p)$ random graph for $p < n^{-1}\log n$, using techniques developed in the case $p \ge (1 + \lambda)n^{-1}\log n$.
Note that when $p < (1 - \lambda)n^{-1}\log n$ for any constant $\lambda > 0$, the number of isolated vertices in $B_0$ has expectation
\begin{equation*}
	|B_0|(1 - p)^{n - 1} = (1 \pm O(p^2))|B_0|e^{-pn} = \Theta(|B_0|n^{\lambda-1}),
\end{equation*}
and variance
\begin{equation*}
	2\binom{|B_0|}{2}p(1 - p)^{2n - 3} + |B_0|(1 - p)^{n - 1} - |B_0|(1 - p)^{2n - 2} = \Theta(|B_0|n^{\lambda-1}).
\end{equation*}
Both quantities are $\Theta(n^\lambda)$ for any setting with $|B_0| = \Omega(n)$, so there are isolated vertices in $B_0$ with probability $1 - o_n(1)$ by a simple Chebyshev bound.
Clearly the same applies for $R_0$, so unanimity is unreachable with probability $1 - o_n(1)$ when the density $p$ is below the connectivity threshold.
However, with a fixed initial gap $\adv$ satisfying $p\adv \ge 10$, we show that the number of Blue vertices at the end is at most a small fraction of $n$, which decreases as $p$ increases, with high probability.
\begin{theorem}  \label{thm:subcon-main}
	Given any constant $0 < \lambda < 1$, consider Majority Dynamics on a $G(n, p)$ random graph with $p = (1 - \lambda)n^{-1}\log n$ and a fixed initial gap $\adv$ such that $p\adv \ge 10$.
	With probability at least
	\begin{equation*}
		1 - O\left(\frac{1}{\min\{n, p\adv^2\}} + n^{-(1 - \lambda)}\right),
	\end{equation*}
	there is some $t^*$ such that $B_t = B_{t^*}$ or $B_t = B_{t^* + 1}$ for all $t\ge t^*$ and
	\begin{equation*}
		\card{B_{t^* + 1}} \le \card{B_{t^*}} \le (\log n)^{O_\lambda(\log^{1/2} n)}n^\lambda,
	\end{equation*}
	and the average degree of vertices in $B_{t^*}$ is $O_\lambda(\sqrt{pn})$.
\end{theorem}
\begin{proof}
	Since we only use the assumption that $p \ge (1 + \lambda)n^{-1}\log n$ in the proof of Theorem \ref{thm:main-nonlazy} at Step 3 (Section \ref{sec:day3+}), in particular, to show that the minimum degree in $G$ is $\Omega(pn)$ with high probability, every other argument still works.
	The arguments in Steps 1 and 2 imply that
	\begin{equation*}
		\Pr\left(\card{B_2} \le .25n\right) \ge 1 - O\left(\frac{1}{\min\{n, p\adv^2\}}\right).
	\end{equation*}
	Assuming $\card{B_2} \le .25n$, the first half of the proof of Lemma \ref{lem:step3} gives for all $t \ge 3$,
	Assuming this, we have, with probability at least $1 - n^{-(1 - \lambda)}$, $G$ has a property that implies for all $t \ge 3$,
	\begin{equation}  \label{eq:subcon-main-tmp1}
		\card{B_{t+1}}\Bigl(\frac{1}{\card{B_{t+1}}}\sum_{v\in B_{t+1}}d(v)\Bigr)^2 \le M_\lambda pn\card{B_t},
	\end{equation}
	for some constant $M_\lambda$ depending only on $\lambda$.
	Recall that the next step in Lemma \ref{lem:step3} is to use the fact that the average degree in $B_{t+1}$ is at least the minimum degree in $G$ and is thus $\Omega(pn)$ with high probability (Lemma \ref{lem:degree-lower-bound}), so Blue shrinks by a factor of $\Omega(pn)$ until Red wins.
	When $p = (1 - \lambda)n^{-1}\log n$, the weaker Lemma \ref{lem:subcon-avg-deg} states that if we can choose $0 < \delta < 1 - \lambda$ and $0 < \eps < 1$ such that $\eps\log\left(\frac{e}{\eps}\right) \le \lambda/4$, then with probability $1 - e^{-\Omega(n^\lambda)}$, $G$ satisfies
	\begin{equation*}
		\frac{1}{\card{U}}\sum_{v\in U}d(v) \ge \eps pn \ \text{ for all } U\subset V \text{ such that } \card{U}\ge n^{\lambda + \delta}.
	\end{equation*}
	Plugging this bound into Eq. \eqref{eq:subcon-main-tmp1}, we have, as long as $\card{B_{t+1}} \ge n^{\lambda + \delta}$,
	\begin{equation*}
		\card{B_{t+1}}(\eps pn)^2 \le M_\lambda pn\card{B_t} \implies \card{B_{t+1}} \le \frac{M_\lambda}{\eps^2 pn}\card{B_t}.
	\end{equation*}
	Let us choose $\eps = \sqrt{2M_\lambda/(pn)} = \Theta((pn)^{-1/2})$ so that the inequality above becomes $\card{B_{t+1}} \le \card{B_t}/2$.
	To satisfy the condition $\eps\log\left(\frac{e}{\eps}\right) < \delta/4$, we can choose
	\begin{equation*}
		\delta = 5\sqrt{M_\lambda}(pn)^{-1/2}\log(pn) = O_\lambda(\log^{-1/2} n\log\log n).
	\end{equation*}
	Therefore the number of Blue vertices will keep halving with probability $1 - e^{-\Omega(n^{\lambda})}$ until, for all sufficiently large $t$,
	\begin{equation*}  \label{eq:subcon-main-tmp2}
		\card{B_t} \le n^{\lambda + \delta} = (\log n)^{O_\lambda(\log^{1/2} n)}n^{\lambda}.
	\end{equation*}
	It is independently proven in \cite{goles1981,poljak1983} that at some point $t^*$, the process will become periodic with period at most 2, i.e. $B_{t^*} = B_{t^* + 2}$.
	Without any loss, we can assume $\card{B_{t^*}} \ge \card{B_{t^* + 1}}$.
	Eq. \eqref{eq:subcon-main-tmp1} again implies that
	\begin{equation*}
		\card{B_{t^*}}\Bigl(\frac{1}{\card{B_{t^*}}}\sum_{v\in B_{t^*}}d(v)\Bigr)^2 \le M_\lambda pn\card{B_{t^* + 1}} \le M_\lambda pn\card{B_{t^*}},
	\end{equation*}
	This implies the average degree in $B_{t^*}$ is at most $\sqrt{M_\lambda pn}$, so Eq. \eqref{eq:subcon-main-tmp2} implies that $\card{B_{t^*}} \le (\log n)^{O_\lambda(\log^{1/2} n)}n^{\lambda}$, thus so is $\card{B_{t^* + 1}}$.
	The proof is complete.
\end{proof}

\begin{remark}
	Theorem \ref{thm:subcon-main} is optimal up to a factor of $(\log n)^{O_\lambda(\log^{1/2} n)}$, since the argument preceding it guarantees there are  $\Omega(n^\lambda)$ Blue vertices remaining with high probability.
	If the process becomes stable, most remaining Blue vertices also have very low degrees ($o(pn)$).
	If the process alternates between $B_{t^*}$ and $B_{t^* + 1}$, although we showed that $B_{t^*}$, the larger set, mostly has low-degree vertices, this may not be true for $B_{t^* + 1}$.
	An example where $B_{t^* + 1}$ has high-degree vertices is a \emph{flower}, a connected component where a single vertex $v$, called the \emph{center} is connected to $\Omega(pn)$ leaves.
	If $B_{t^* + 1}$ contains mostly centers of flowers, while $B_{t^*}$ contains mostly leaves, the average degree in $B_{t^* + 1}$ will be $\Omega(pn)$.
\end{remark}

\section{Random Activations and Updates: Proof of Theorem \ref{thm:main-lazy}} \label{sec:main-lazy}

We show the theorem below, which can be readily seen to imply Theorem \ref{thm:main-lazy}.

\begin{theorem}  \label{thm:main-lazy-details}
	Consider the majority dynamics process on $G(n, p)$ with activation rate $\pac$ and updating rate $\pup$.
	Assume that $\card{R_0} = n/2 + \adv$ and $\pup$, $\pac$ and $p$ satisfy $p\pac n \ge (1 + \lambda)\log n$ for a constant $\lambda > 0$, and
    $p\pup\pac^2\adv \ge M$ for a constant $M > 0$.
	Then a Red unanimity is reached within time $O(\pup^{-1}\pac^{-1}\log n)$ with probability at least
	\begin{equation*}
		1 - O\left(\frac{1}{\pup^2\pac^2\min\{n, p\pac \adv^2\}}\right) - e^{-\Omega(\pup^2\pac^4p\adv^2)} - O(n^{-\lambda/4}) - n^{-\Omega(\pup\pac)}.
	\end{equation*}
\end{theorem}

We modify the proof of Theorem \ref{thm:main-nonlazy} slightly to obtain the proof of Theorem \ref{thm:main-lazy-details}.
To prepare for the technical additions, we define several new notations.
\begin{itemize}
	\item For a given set $S\subset V$, let $\upd{S}$ denote the set of updating vertices in $S$.
	If $S_t$ is a set that changes with time, let $\upd{S_t}$ denote the set of updating vertices in $S_t$ on day $t$.
	The notations $\ac{S}$ and $\ac{S_t}$ are
    analogously
    defined for active vertices. 
	\item Given a coloring $\Clr$, let $\hypo{\Clr}$ be the hypothetical coloring the next day \emph{if} every node is active and updating, i.e. $\hypo{\Clr}_t = \Clr_{t+1}$ in the base majority dynamics model.
	Let $\hypo{R}$ and $\hypo{B}$ be the respective Red and Blue sides under $\hypo{\Clr}$.
	\item $u \obs v$ ($u \obs_t v$) denotes the event $u$ observes $v$ (on day $t$).
\end{itemize}

Our proof consists of three main propositions.
\begin{proposition}  \label{prop:day1-lazy}
	Consider majority dynamics under the setting of Theorem \ref{thm:main-lazy-details}.
    For every $\eps \in (0, 1)$ and $T > 0$, there are $M_K$ and $S_\eps$, $T_\eps$ such that, when
    \[
    \pac^{1/2}\adv \ge S_\eps n^{1/2},
    \quad p^{1/2}\pup\pac\adv \ge T_\eps,
    \quad \text{ and } \quad p\pup\pac^2\adv \ge M_K,
    \]
    we have, with probability at least $1 - \eps$,
    \begin{equation} \label{eq:main-lazy-details-temp1}
        \card{B_1} \le \frac{n}{2} - K\sqrt{\frac{n}{p\pac}}.
    \end{equation}
\end{proposition}



The proof of this proposition is almost identical to that of Lemma \ref{lem:day1} and will be put in Appendix \ref{sec:shrink-scheme}.

At this point we can assume that $\adv$ is small enough so that $\adv\sqrt{p\pac n} \le n$ and $\card{B_1} \ge \frac{n}{2} - .2 \pup\pac \adv\sqrt{p\pac n} \ge (1/2 - .2\pup\pac)n \ge .3n$.
Larger values of $\adv$ or smaller $\card{B_1}$ can only increase the probability of Red winning.


\begin{proposition}  \label{prop:main-lazy-details-prop2}
	Under the assumptions of Theorem \ref{thm:main-lazy-details} and the assumption \eqref{eq:main-lazy-details-temp1}, for every $\eps > 0$ and $K \ge 8$ there is $t^*_1 = O\left(\pup^{-1}\pac^{-1}\log n\right)$ such that, with probability at least
	\begin{equation*}
		1 - e^{-\Omega(K^2p^{-1})} - n^{-\Omega(\pup\pac)},
	\end{equation*}
	we have, for $\alpha_\lambda$ defined in Lemma \ref{lem:degree-lower-bound},
	\begin{equation}  \label{eq:main-lazy-details-temp2}
		\card{B_{t^*_1}} \le \alpha_\lambda(\log n)/2.
	\end{equation}
\end{proposition}

\begin{proof}
	For each $t\ge 2$, we split the next-day Blue camp into two disjoint camps, the non-updating currently Blues, and the updating hypothetically Blues, so that 
	\begin{equation}  \label{eq:main-lazy-details-temp0}
		\card{B_{t+1}} = \card{B_t\setminus \upd{B_t}} + \card{\upd{\hypo{B}_t}}.
	\end{equation}
	The former is a $\Bin(\card{B_t}, 1 - \pup\pac)$ variable, so by a standard Chernoff bound (e.g. Lemma \ref{lem:gen-chernoff}), we have, with probability $1 - e^{-\Omega(\pup\pac n)}$,
	\begin{equation}  \label{eq:main-lazy-details-inupdating-bound}
		(1 - 1.1\pup\pac)\card{B_t} \le \card{B_1\setminus \upd{B_t}} \le (1 - .9\pup\pac)\card{B_t}.
	\end{equation}
	To bound the latter, note that $\card{\upd{\hypo{B}_t}} \sim \Bin(\card{\ac{\hypo{B}_t}}, \pup)$,
	where $\ac{\hypo{B}_t}$ can be viewed as the next day's Blue camp inside the subgraph $G(\ac{V_t})$ in the normal case, which can be shown to shrink from $\ac{B_t}$ using Lemma \ref{lem:step2} (for $t = 2$) or Lemma \ref{lem:hypo-blue-bound-2} (for $t \ge 3$).
	Thus we can find a bound on $\card{\ac{\hypo{B}_t}}$, which then leads to a bound on $\card{\upd{\hypo{B}_t}}$ with sufficiently high probability.
	The exact details differ for Day 2 and Day $t\ge 3$, just like the normal case, so we split this step into two subcases:
	
	\textit{Case 1. Bounding $\card{\upd{\hypo{B}_2}}$.}	
	Let $M \defeq Kp^{-1/2}/4 \le \sqrt{\pac n}$.
	By a Chernoff bound (Lemma \ref{lem:gen-chernoff}), we have, with probability $1 - e^{-\Omega(M^2)} = 1 - e^{-\Omega(K^2p^{-1})}$,
	\begin{equation*}
        \card{\ac{B_1}} = \pac\card{B_1} \pm T\sqrt{\pac\card{B_1}}, \quad
			\card{\ac{V_1}} = \pac n \pm T\sqrt{\pac n} \le 1.1\pac n.
	\end{equation*}
	Assuming the above, we have
	\begin{equation*}
    \begin{aligned}
			& \frac{\card{\ac{V_1}}}{2} - \card{\ac{B_1}}
			\ge \pac\left(\frac{n}{2} - \card{B_1}\right) - 2T\sqrt{\pac n} \\
			& \ge K\sqrt{\frac{\pac n}{p}} - 2T\sqrt{\pac n}
            = \frac{K}{2}\sqrt{\frac{\pac n}{p}} \ge \frac{K}{2}\sqrt{\frac{\card{\ac{V_1}}}{1.1p}}
            \ge 2\sqrt{\frac{\card{\ac{V_1}}}{p}}.
	\end{aligned}
	\end{equation*}
 
	By Lemma \ref{lem:step2} in the normal case, with probability $1 - e^{-\Omega(\pac n)}$,
	the subgraph $G(\ac{V_1})$ is such that $\card{\ac{\hypo{B}_2}} \le .3\card{\ac{V_1}} \le .33\pac n$.
	By a generalization of Chernoff bound (Lemma \ref{lem:gen-chernoff}), we have $\card{\upd{\hypo{B}_2}} \le 1.05\card{\ac{\hypo{B}_2}} \le .35\pup\pac n$ with probability $1 - \exp(-\Omega(\pup\pac n))$.
	In total, with probability
	\begin{equation*}
		1 - e^{-\Omega(\pup\pac n)} - e^{-\Omega(\pac n)} - e^{-\Omega(K^2p^{-1})}
		= 1 - e^{-\Omega(\pup\pac n)} - e^{-\Omega(K^2p^{-1})},
	\end{equation*}
	we can plug this bound on $\card{\upd{\hypo{B}_2}}$ and Eq. \eqref{eq:main-lazy-details-inupdating-bound} into Eq. \eqref{eq:main-lazy-details-temp0} to get
	\begin{equation}  \label{eq:main-lazy-details-temp3}
		\card{B_2} \le .35\pup\pac n + (1 - .9\pup\pac)\card{B_1} \le \left(\frac{1}{2} - .1\pup\pac\right)n.
	\end{equation}
	
	\textit{Case 2. Bounding $\card{\upd{\hypo{B}_t}}$ for $t \ge 3$ (until $\card{B_t} < \alpha_\lambda(\log n)/2$).}
	Assume $\card{B_t} \le bn$ for some $t\ge 2$ and $\alpha_\lambda(\log n)/2 \ge b \ge 1/2 - .1\pup\pac$.
	With probability $1 - e^{-\Omega(\pac\log n))}$, we have
	\begin{equation*}
		\begin{aligned}
			\card{\ac{B_t}} & = \pac \card{B_t} \pm \sqrt{\pac \card{B_t} (\pac \log n)} = \pac \card{B_t} \pm \pac \sqrt{b n\log n} \\
			\card{\ac{V_t}} & = \pac n \pm \sqrt{\pac n (\pac \log n)} = \pac n \pm \pac \sqrt{n \log n}.
		\end{aligned}
	\end{equation*}
	Therefore,
	\begin{equation*}
		\frac{\card{\ac{B_t}}}{\card{\ac{V_t}}} = \frac{\pac \card{B_t} \pm \pac \sqrt{b n\log n}}{\pac n \pm \pac \sqrt{n \log n}}
		= \frac{\card{B_t}}{n} \pm \frac{\sqrt{b\log n}}{\sqrt{n}} = b'.
	\end{equation*}
	When $b \le 1/2 - .1\pup\pac$, for $n$ sufficiently large, the above implies
	\begin{equation*}
		\frac{\card{\ac{B_t}}}{\card{\ac{V_t}}} \le b' \le \frac{1}{2} - .09\pup\pac.
	\end{equation*}
	By Lemma \ref{lem:step3}, with probability $1 - O(n^{-1})$, the subgraph $G(\ac{V_t})$ satisfies, for some universal constant $C$,
	\begin{equation*}
		\card{\ac{\hypo{B}_t}} \le \frac{C}{p\card{\ac{V_t}}} \frac{(1 - b')\card{\ac{B_t}}}{(1/2 - b')^2} 
		\le \frac{C(1/2 + .09\pup\pac)\card{\ac{B_t}}}{(.09)^2\pup^2\pac^2pn}  \le \frac{C'\pac \card{B_t}}{pn} \le \frac{\pac \card{B_t}}{20}.
	\end{equation*}
	By Lemma \ref{lem:gen-chernoff}, with probability $1 - e^{-\Omega(\pup\pac\card{B_t})}$, $\card{\upd{\hypo{B}_t}} \le .1\pup\pac \card{B_t}$.
	In total, with probability at least
	\begin{equation*}
		1 - e^{-\Omega(\pac\log n)} - e^{-\Omega(\pup\pac\card{B_t})} - O(n^{-1})
		= 1 - n^{-\Omega(\pup\pac)},
	\end{equation*}
	we can plug this bound on $\card{\upd{\hypo{B}_t}}$ and Eq. \eqref{eq:main-lazy-details-inupdating-bound} into Eq. \eqref{eq:main-lazy-details-temp0} to get
	\begin{equation}  \label{eq:gen-temp1}
		\card{B_{t+1}} 
		\le .1\pup\pac\card{B_t} + (1 - .9\pup\pac)\card{B_t} = (1 - .8\pup\pac)\card{B_t}.
	\end{equation}
	Let $t^*_1$ be the first time for which $(1 - .8\pup\pac)^{t^*_1 - 2}\card{B_2} < \alpha_\lambda(\log n)/2$, so
	\begin{equation*}
		t^*_1 = O(\log_{1 - .8\pup\pac}n) = O\left(\frac{\log n}{\pup\pac}\right).
	\end{equation*}
	By a union bound, the probability that \eqref{eq:gen-temp1} holds for all $t\in [2, t^*_1-1]$, meaning $\card{B_{t^*_1}} < \alpha_\lambda(\log n)/2$, given that $\card{B_2} \le (1/2 - .1\pup\pac)n$, is at least
	\begin{equation*}
		1 - \frac{\log n}{\pup\pac}n^{-\Omega(\pup\pac)} = 1 - n^{-\Omega(\pup\pac)}.
	\end{equation*}
	Including the probability that the condition $\card{B_2} \le (1/2 - .1\pup\pac)n$ does not hold, we get the probability bound in the proposition.
\end{proof}

\begin{proposition}
	Under the assumptions of Theorem \ref{thm:main-lazy-details} and the assumption \eqref{eq:main-lazy-details-temp2}, there is $t^*_2 = t^*_1 +  O\left(\log n\right)$ such that Red wins after Day $t^*_2$ with probability at least
	$1 - O(n^{-\lambda/4}) - O(n^{-\pup\pac/2})$.
\end{proposition}

\begin{proof}
	Temporarily fix an arbitrary time period $T$.
	For each day $t$ between $t^* + 1$ and $t^* + T$, by Lemma \ref{lem:degree-lower-bound}, the subgraph $G(\ac{V_t})$ satisfies
	\begin{equation}  \label{eq:main-lazy-details-temp5}
		\min\left\{d_{\ac{V_t}}(v) : v\in \ac{V_t} \right\} \ge \alpha_\lambda p\pac n
	\end{equation} with probability at least $1 - (\pac n)^{-\lambda/2} = 1 - O(n^{-\lambda/3})$.
	If this happens for all $t\in [t^*_1 + 1, t^* + T]$, any remaining Blue vertices that update at least once during this period becomes Red until dat $t^* + T$ because they will see more Red neighbors than twice the amount of possible Blue neighbors.
	The probability a fixed vertex is not updating for the whole period is $(1 - \pup\pac)^T$.
	By union bounds, the event that Eq. \eqref{eq:main-lazy-details-temp5} holds for all $t\in [t^*_1 + 1, t^* + T]$ and every vertex in $B_{t^*}$ updates at least once occurs with probability at least
	\begin{equation*}
		1 - O(Tn^{-\lambda/3}) - O((1 - \pup\pac)^T\log n).
	\end{equation*}
	Choosing $T \ge \log n$, this probability is $1 - O(n^{-\lambda/4}) - O(n^{-\pup\pac/2}) = 1 - o(1)$.
	Thus for $t^*_2 = t^*_1 + O(\log n)$, $B_{t^*_2} = \varnothing$ with the above probability.
\end{proof}

\begin{remark}
    When $\pup$ and $\pac$ are not both 1, meaning $\pup\pac < 1 - \eps$ for some $\eps > 0$, we can use the Chernoff bound in a similar manner as in Eq. \eqref{eq:main-lazy-details-inupdating-bound} to show that, with probability $1 - o(1)$, $\card{B_{t+1}} \ge (1 - (1 + \eps)\pup\pac)\card{B_t}$ for each $t\le n$, so Blue takes at least $\Omega(\log_{1 - (1 + \eps)\pup\pac}(1/n)) = \Omega(\pup^{-1}\pac^{-1}\log n)$ days to vanish, meaning the convergence time in Theorem \ref{thm:main-lazy} is asymptotically optimal.
\end{remark}

\section*{Acknowledgement}
We would like to thank Q. A. Do, A, Ferber, A. Deneanu, J. Fox, E. Mossel,  and X. Chen for inspiring discussions; and H. Le, T. Can, and L. T. D. Tran for proofreading the paper.

\appendix

\section {Some Preliminary Results in Probability} \label{sec:prelim-lemmas}

We provide some supporting results that are used throughout the proof of the main theorems in
Sections \ref{sec:main-nonlazy-proof} 
and \ref{sec:main-lazy}.

Berry-Esseen theorem in its classical form establishes an explicit bound for the convergence rate of sums of random variables to the normal distribution.
\begin{theorem}[Berry-Esseen] \label{thm:esseen}
	There is a universal constant $\Cbe$, such that for any $n\in \N$, if $X_1, X_2, X_3, \cdots, X_n$ are random variables with zero means, variances $\sigma_1, \sigma_2, \cdots, \sigma_n > 0$, and absolute third moments $\E{|X_i|^3} = \rho_i < \infty$, we have:
	\begin{equation*} \label{BE1}
		\sup_{x\in \bb{R}}\abs{\Pr\left( \frac{\sum_{i=1}^nX_i}{\sigma_X} \le x \right) - \Phi(x)} \le \Cbe\sum_{i=1}^n\rho_i\left(\sum_{i=1}^n\sigma_i^2\right)^{-3/2}
	\end{equation*}
\end{theorem}

The original proof by Esseen \cite{esseen} yielded  $\Cbe = 7.59$, and this constant has been improved a number of times. The latest work by Shevtsova \cite[Theorem 1]{shevtsova} achieved  $\Cbe = .56$, which
is used throughout this
paper.
We will be interested in the setting where   $\{X_i\}_{i=1}^n$ are r.v.s taking values in either $\{0, 1\}$ or $\{0, -1\}$. The following corollary and lemmas naturally follow from this theorem.

\begin{corollary} \label{lem:strongesseen}
	Let $p\in (0, 1)$ and $X_1, X_2, \cdots, X_n$ be Bernoulli random variables such that for all $i$, either
	$X_i \sim \Ber(p)$ or $X_i \sim -\Ber(p)$. Let
	$X = X_1 + X_2 + \cdots + X_n$ and $\mu_X = \E{X}$.
	Then,
	\begin{equation*}  \label{BE3}
		\sup_{x\in \bb{R}}\abs{\Pr\left(X - \mu_X \le x\right) - \Phi\left(\frac{x}{\sigma\sqrt{n}}\right)}
		\le \frac{\Cbe(1 - 2\sigma^2)}{\sigma\sqrt{n}} \quad \text{ where } \sigma = \sqrt{p(1-p)}.
	\end{equation*}
\end{corollary}



This implies the following lemma, which bounds the probability that two binomial variables are \textit{exactly} some units apart.
\begin{lemma} \label{lem:esseen-err-term}
	Let $p \in (0, 1)$ be a constant and $\sigma = \sqrt{p(1-p)}$,  $X_1 \sim \Bin(n_1, p)$ and $X_2 \sim \Bin(n_2, p)$ be independent
    random variables defined on the same probability space.
    Then for any positive integer $d < (n_1 + n_2)/2$,
	\begin{equation*}
		\Pr\left(X_1 = X_2 + d \right) \le \frac{1.12\left(1 - 2\sigma^2\right)}{\sigma\sqrt{n_1 + n_2}}.
	\end{equation*}
\end{lemma}

\begin{lemma} \label{lem:degree-lower-bound}
	Let $\eps$ be an arbitrary positive constant, $N_\eps = 4/\eps + 4$,
	and $\alpha_\eps = \exp(-(8/\eps + 7))$.
	Then with probability at least $1 - n^{-\eps/2}$, a $G(n, p)$ random graph with $n\ge N_\eps$ and $pn \ge (1 + \eps)\log n$ has minimum degree at least $\alpha_\eps pn$. 
\end{lemma}

\begin{proof}
	First, let $\alpha$ be arbitrary.
	Using a Chernoff bound, given that
        $d(v) \sim \Bin(n - 1, p)$, we have
    \[
	\Pr(d(v) \le \alpha p(n - 1)) \le \exp\Bigl[-(n - 1) D(\alpha p \ \|\ p)\Bigr],
	\]
    where
    \[
    D(x\ \| \ y) \defeq x\log\left(\frac{x}{y}\right) + (1 - x)\log\left(\frac{1 - x}{1 - y}\right).
    \]
	It can be shown easily that
    \[
	D(\alpha p \ \| \ p) = p\left[\alpha \log\alpha + \frac{1 - \alpha p}{p}\log\left(\frac{1 - \alpha p}{1 - p}\right)\right] \ge p(1 - \alpha + \alpha\log \alpha).
	\]
	Therefore
	\[
	\Pr(d(v) \le \alpha p(n - 1)) \le \exp\Bigl[-p(n - 1) (1 - \alpha + \alpha\log \alpha)\Bigr].
	\]
	Replacing $\alpha$ with $\alpha \frac{n}{n-1}$, we get
	\[
	\begin{aligned}
		\Pr(d(v) \le \alpha pn)
		& \le \exp\left[-pn \Bigl(1 - \frac{1}{n} - \alpha + \alpha\log\left(\frac{\alpha n}{n - 1}\right)\Bigr)\right] \\
		& \le \exp\left[-pn \Bigl(1 - \alpha + \alpha\log \alpha - \frac{1}{n}\Bigr) \right]
		.
	\end{aligned}
	\]
	Applying a union bound over all choices of $v$, we get
	\[
	\Pr(\exists v, d(v)\le \alpha pn) \le n\cdot \exp\left[-pn \Bigl(1 - \alpha + \alpha\log \alpha - \frac{1}{n}\Bigr) \right] = e^{-T},
	\]
	where
	$
	T \defeq pn \bigl(1 - \alpha + \alpha\log \alpha - \frac{1}{n}\bigr) - \log n \ge T'\log n$, where
	\[
	T'\defeq (1 + \eps)\bigl(1 - \alpha + \alpha\log \alpha - n^{-1}\bigr) - 1
	= \eps - (1 + \eps)\alpha(1 - \log \alpha ) - (1 + \eps)/n.
	\]
	Now assume $\alpha = \alpha_\eps = \exp(-(8/\eps + 7))$, and let $t = -\log \alpha \ge 8/\eps + 7$.
	Then
	\[
	\alpha(1 - \log \alpha) = e^{-t}(1 + t) \le 8(1/\eps + 1)\exp\left(-8/\eps - 7\right) < \tfrac{1}{4}\eps/(1 + \eps).
	\]
	Assuming $n \ge N_\eps = 4/\eps + 4$, we get $(1 + \eps)n^{-1} \le \eps/4$.
	The two inequalities above give $T' \ge \eps - \eps/4 - \eps/4 = \eps/2$.
	By a union bound, the probability that $d(v)\le \alpha_\eps pn$ for some $v$ is at most
	$
	e^{-T'\log n} \le n^{-\eps/2}
	$.
    This completes the proof.
\end{proof}

When $p$ is below the connectivity threshold, we have a weaker version of the above lemma, which gives a lower bound on the average degree of sufficiently large subsets of $V$.
\begin{lemma}  \label{lem:subcon-avg-deg}
	Consider a $G(n, p)$ random graph with $p = (1 - \lambda)n^{-1}\log n$.
	For any $0 < \delta < 1 - \lambda$ and any $0 < \eps < 1$ such that $\eps\log\left(\frac{e}{\eps}\right) \le \delta/4$, we have: with probability at least $1 - n^{-\delta n^{\lambda + \delta}/2}$, any subset $U\subset V$ such that $\card{U}\ge n^{\lambda + \delta}$ has average degree at least $\eps pn$.
\end{lemma}

\begin{proof}
	Fix an arbitrary subset $U$ with 
    $|U| = \lceil n^{\lambda + \delta} \rceil$
    .
	Consider
	\begin{equation*}
		X \defeq \sum_{u\in U}d(u) = \sum_{u\in U}\sum_{v\in V}\one_{u\adj v}
		= 2\sum_{\{u, v\}\subset U}\one_{u\adj v} + \sum_{u\in U}\sum_{v\in U^c}\one_{u\adj v}.
	\end{equation*}
	Let $X_1$ denote the first sum, and $X_2$ denote the second double sum.
	Then $X_1$ and $X_2$ are indpendent binomial variables, where $X_1\sim \Bin\bigl(\binom{|U|}{2}, p\bigr)$ and $X_2 \sim \Bin\left(|U||U^c|, p\right)$.
	In the context of this proof, let $f_Y(t) \defeq t\eps\E{Y} + \log\E{e^{-tY}}$ for a random variable $Y$.
	A Chernoff bound gives
	\begin{equation*}
		\Pr\left(X\le \eps \E{X}\right)
		\le e^{f_X(t)} = e^{f_{2X_1 + X_2}(t)} = e^{f_{X_1}(2t) + f_{X_2}(t)}.
	\end{equation*}
    Let $f(t) \defeq f_{\Ber(p)}(t) = t\eps p + \log(pe^{-t} + 1 - p)$.
	When $Y \sim \Bin(m, p)$, we have $f_Y(t) = mf(t)$, so 
	\begin{equation*}
		f_{X_1}(2t) + f_{X_2}(t)
		= \binom{|U|}{2}f(2t) + |U||U^c|f(t).
	\end{equation*}
	It is well-known in the proof of Chernoff's inequality that $t_* = \log\left(\frac{1 - p\eps}{\eps(1 - p)}\right)$ minimizes $f(t)$ at
	$f(t_*) = -D(p\eps \ \| \ p) = -p\eps\log\eps - (1 - p\eps)\log\left(\frac{1 - p\eps}{1 - p}\right)$.
	We have
	\begin{equation*}
		\begin{aligned}
			& f(2t_*) - 2f(t_*)
			= \log\left(p\frac{\eps^2(1 - p)^2}{(1 - p\eps)^2} + 1 - p\right) - 2\log\left(p\frac{\eps(1 - p)}{1 - p\eps} + 1 - p\right), \\
			& = \log\left(\frac{(1 - p)(1 - 2p\eps + p\eps^2)}{(1 - p\eps)^2}\right) - 2\log\left(\frac{1 - p}{1 - p\eps}\right)
			= \log\left(\frac{1 - 2p\eps + p\eps^2}{1 - p}\right) = O(p).
		\end{aligned}
	\end{equation*}
	Thus
	\begin{equation*}
		\begin{aligned}
			& \binom{|U|}{2}f(2t_*) + |U||U^c|f(t_*)
			= \left[2\binom{|U|}{2} + |U||U^c|\right]f(t_*) + O(p|U|^2) \\
			& = |U|\left[-(n - 1)f(t_*) + O(p|U|)\right] = -|U|\left[nD(p\eps \ \| \ p) - O(p|U|)\right],
		\end{aligned}
	\end{equation*}
	which is close to the true minimum of the expression.
    We can replace $|U|$ directly with $n^{\lambda + \delta}$ without significantly affecting the computations that follows.
    The last expression becomes
	\begin{equation*}
		-n^{\lambda + \delta}\bigl[nD(p\eps \| p) - O(pn^{\lambda + \delta})\bigr]
		= -n^{1 + \lambda + \delta}\bigl[D(p\eps \| p) - O(pn^{\lambda + \delta - 1})\bigr].
	\end{equation*}
	Let $\mathfrak{E}$ be the event that there is a subset $U$ with at least $n^{\lambda + \delta}$ vertices whose average degree is at least $\eps pn$.
	Then a union bound gives
	\begin{equation}  \label{eq:subcon-avg-deg-tmp1}
		\begin{aligned}
			& \Pr(\mathfrak{E})
			\le \binom{n}{n^{\lambda + \delta}} \exp\Bigl[-n^{1 + \lambda + \delta}\bigl(D(p\eps \| p) - O(pn^{\lambda + \delta - 1})\bigr)\Bigr] \\
			& \le \exp\Bigl[- nh(n^{\lambda + \delta - 1}) -n^{1 + \lambda + \delta}\bigl(D(p\eps \| p) - O(pn^{\lambda + \delta - 1})\bigr)\Bigr],
		\end{aligned}
	\end{equation}
    where we again replace $\lceil n^{\lambda + \delta} \rceil$ with just $n^{\lambda + \delta}$ without significantly changing the calculations.
	We use the following estimates:
	\begin{equation*}
		D(p\eps\|p) = p\eps\log \eps - (1 - p\eps)\log\left(1 - \frac{p(1 - \eps)}{1 - p\eps}\right)
		= p\bigl(\eps\log \eps - \eps + 1\bigr) + O(p^2),
	\end{equation*}
	and
	\begin{equation*}
		h(x) = x\log x - (1 - x)\log\left(\frac{1}{1 - x}\right) = x\log x - x + O(x^2) \text{ for } |x| < 1.
	\end{equation*}
	We have
	\begin{equation*}
		nh(n^{\lambda + \delta - 1}) 
		= n^{\lambda + \delta}(\lambda + \delta - 1)\log n - n^{\lambda + \delta} + O(n^{2(\lambda + \delta) - 1})
	\end{equation*}
	and
	\begin{equation*}
		\begin{aligned}
			& n^{1 + \lambda + \delta}\bigl(D(p\eps \| p) - O(pn^{\lambda + \delta - 1})\bigr) \\
			& = pn^{1 + \lambda + \delta}(1 - \eps + \eps\log\eps) + O(p^2n^{1 + \lambda + \delta}) - O(pn^{2(\lambda + \delta)}) \\
			& = (1 - \lambda)n^{\lambda + \delta}\left(1 - \eps\log\left(\frac{e}{\eps}\right)\right)\log n - O(n^{2(\lambda + \delta) - 1}\log n ).
		\end{aligned}
	\end{equation*}
	The exponent in Eq. \eqref{eq:subcon-avg-deg-tmp1} becomes
	\begin{equation*}
		\begin{aligned}
			& -n^{\lambda + \delta}\left[\left((1 - \lambda)\left(1 - \eps\log\left(\frac{e}{\eps}\right)\right)- 1 + \lambda + \delta\right)\log n  - 1\right] - O(n^{2(\lambda + \delta) - 1}\log n) \\
			& = -n^{\lambda + \delta}\left[\left(\delta - (1 - \lambda)\eps\log\left(\frac{e}{\eps}\right) \right)\log n - 1\right] - O(n^{2(\lambda + \delta) - 1}\log n).
		\end{aligned}
	\end{equation*}
	As given, $\eps$ is small enough so that $(1 - \lambda)\eps\log\left(\frac{e}{\eps}\right) < \delta/4$, so the exponent in Eq. \eqref{eq:subcon-avg-deg-tmp1} is at most
	\begin{equation*}
		-n^{\lambda + \delta}\left(\frac{3\delta}{4} - O(n^{-(1 - \lambda - \delta)})\right)\log n \le  -\frac{1}{2}n^{\lambda + \delta}\delta\log n
	\end{equation*}
	Therefore all subsets of size at least $n^{\lambda + \delta}$ has average degree at least $\eps pn$ with probability at least $1 - n^{-\delta n^{\lambda}/2}$.
\end{proof}

The next lemma is an extension of
the Chernoff bounds
for binomial random variables whose size is not fixed.

\begin{lemma} \label{lem:gen-chernoff}
	Given $p\in (0, 1)$, $m\in (0, \infty)$, if $Y$ be a random variable supported on $[0, m]\cap \Z$
	and $X\sim \Bin(Y, p)$, then for all $t > 0$,
	\[
	\Pr(X\ge pm + t) \le \exp\left[-\max\Bigl\{\frac{t^2}{2(pm + t)}, \frac{t^2}{2(1 - p)m}\Bigr\}\right].
	\]
\end{lemma}

\begin{proof}
    Recall the usual Chernoff bounds for $Z\sim \Bin(m, p)$:
    \begin{align} 
    \label{eq:gen-chernoff-temp3} \Pr(Z\ge T) & \le \exp\left[-\frac{(T - pm)^2}{2T}\right] \quad \text{ for any } T > pm,
    \\
    \label{eq:gen-chernoff-temp4} \Pr(Z\le T) & \le \exp\left[-\frac{(pm - T)^2}{2pm}\right] \quad \text{ for any } T < pm.
    \end{align}
    Using Eq. \ref{eq:gen-chernoff-temp4}, one can get an extension of Eq. \ref{eq:gen-chernoff-temp3} as follows:
	For $T > pm$, consider $Z' \defeq m - Z \sim \Bin(m, 1 - p)$, then $(1 - p)m - (m - T) = T - pm$, so we have
    \begin{equation} \label{eq:gen-chernoff-temp5}
        \Pr(Z\ge T) = \Pr(Z'\le m - T)
        \le
        \exp\left[-\frac{(T - pm)^2}{2(1 - p)m}\right].
    \end{equation}
    Combining Eeqs. \eqref{eq:gen-chernoff-temp3} and \eqref{eq:gen-chernoff-temp5}, we get the desired inequality for the random variable $Z$.
    It now suffices to show that $Z$ \emph{stochastic dominates} $X$, in the sense that for all $t \ge 0$, we have
    \begin{equation*}
        \Pr(Z \le t) \le \Pr(X \le t).
    \end{equation*}
    To prove this, it suffices to couple $X$ and $Z$ in a probability space such that $X'\le Z'$ always holds, where $X'$ and $Z'$ are respectively copies of $X$ and $Z$.
    We can use the following straightforward coupling:
    
    Let $X'$ be an arbitrary copy of $X$.
    Let $Y_1 \defeq m - Y$ and let $X_1 \sim \Bin(Y_1, p)$, independently of $Y$ and $X'$.
    Let $Z' \defeq X' + X_1$.
    For any $k\in [0, m]\cap \Z$, $Z'\sim \Bin(m, p)$ on the subspace conditioned on $Y = k$.
    For each $t\in [0, m]\cap \Z$, we have
    \begin{equation*}
    \begin{aligned}
        \Pr(Z' = t)
        & = \sum_{k = 0}^m \Pr(Z' = t \mid Y = k)\Pr(Y = k)
        \\
        & = \sum_{k = 0}^m \binom{m}{t} p^t (1 - p)^{m - t}\Pr(Y = k)
        =  \binom{m}{t} p^t (1 - p)^{m - t}.
    \end{aligned}
    \end{equation*}
    Therefore $Z'\sim \Bin(m, p)$.
    Since $Z' \ge X'$ always from their definitions, the proof is complete.
\end{proof}

The final lemma comes from a breakthrough work by Georges, Bordenave and Knowles \cite{georges2017}, which provides a tight bound on the operator norm of a random matrix with independent but arbitrary, potentially non-identical cell distributions.
Their result extends another pioneering work by Bandeira and van Handel \cite{bandeira2014} for the case of non-sub-Gaussian cell distribution.
It can be proven using \cite[Theorem 3.2]{georges2017} and some techniques in \cite{bandeira2014}.

\begin{lemma} \label{lem:matrix-norm-bound}
	Let $c, a > 0$ be constants.
    Assume $pn \ge c\log n$.
	With probability at least $1 - n^{-a}$, the adjacency matrix $A$ of a random graph $G\sim G(n, p)$ satisfies
	$\|A - \E{A}\|_{op} \le C_{c, a}\sqrt{pn}$ for a constant $C_{c, a}$ depending on $c$ and $a$.
\end{lemma}

\begin{proof}
    Let $X \defeq A - \E{A}$.
    Applying \cite[Theorem 3.2]{georges2017} for the simple case $p_{ij} = p\one_{\{i \neq j\}}$ for all $i, j\in [n]$, we get
    \begin{equation*}
        \E{\|X\|_{op}} \le \left(2 + C\eta \right)\sqrt{d},
    \end{equation*}
    where $C$ is a universal constant, $d \defeq p(1 - p)(n - 1) \le pn$, and
    \begin{equation*}
        \eta \defeq \frac{\sqrt{\log n}}{\min\{\sqrt{d}, n^{.1}\}}
        \le \frac{\sqrt{\log n}}{\sqrt{c\log n}} = c^{-1/2}.
    \end{equation*}
    Therefore, we obtain a simpler bound:
    \begin{equation*}
        \E{\|X\|_{op}} \le (2 + Cc^{-1/2})\sqrt{pn}.
    \end{equation*}
	Finally, a Talagrand-like inequality \cite[Theorem 6.10]{boucheron2013} for the convex $1$-Lipschitz function $g: [0, 1]^{n(n-1)/2} \to \R$: $\{Z_{ij}\}_{i<j} \mapsto 2^{-1/2}\|Z - \E{A}\|_{op}$ gives
	\[
	\Pr\left(\|X\|_{op} \ge \E{\|X\|_{op}} + t\right) \le e^{-t^2/4}.
	\]
	Combining the three inequalities, we get
	\begin{equation*}
	\Pr\left(\|X\|_{op} \ge (2 + Cc^{-1/2})\sqrt{pn} + t\right) \le e^{-t^2/4}.
	\end{equation*}
	Letting $t = 2\sqrt{apn}$, we get the desired bound for $C_{c, a} \defeq 2 + Cc^{-1/2} + 2\sqrt{a}$.
\end{proof}

\section{Proof of Technical Results in Main Theorem's Proof} \label{sec:shrink-scheme}


In this section, we present the proofs of
\ref{lem:hypo-blue-bound-2} in the last two steps of the proof of Theorem \ref{thm:main-nonlazy}, and Proposition \ref{prop:day1-lazy} in the proof of Theorem \ref{thm:main-lazy}.



\begin{proof}[Proof of Lemma \ref{lem:hypo-blue-bound-2}]
	Let $b_t \defeq \card{B_t}/n$ and 
	$\mathbf{v} \defeq \bigl[\chi_{B_t}(u) - b_t\bigr]_{u\in V} \in \R^n$.
	We have
	\[
	\mathbf{1}_n^T \mathbf{v} = 0 \ \text{ and } \ \|\mathbf{v}\|_2^2 = b_tn(1 - b_t)^2 + (1 - b_t)nb_t^2 = b_t(1 - b_t)n.
	\]
	Now for any $x$, we have
	\[
	\|A_G\mathbf{v}\|_2^2 = \|(A_G - x\mathbf{1}_n\mathbf{1}_n^T)\mathbf{v}\|_2^2
	\le \bigl\|A_G - x\mathbf{1}_n\mathbf{1}_n^T\bigr\|_{op}^2 \cdot (1 - b_t)\card{B_t}.
	\]
	The LHS has the form
	\[
	\|A_G\mathbf{v}\|_2^2 = \sum_{v\in V} \Bigl[ \sum_{u\adj v}(\chi_{B_t}(u) - b_t) \Bigr]^2
	= \sum_{v\in V} (d_{B_t}(v) - b_td(v))^2. 
	\]
	Considering the fact $v\in B_{t+1}$ satisfies $d_{B_t}(v) \ge \tfrac{1}{2}d(v) \ge b_td(v)$, we have
	\[
	\begin{aligned}
		& \sum_{v\in V} (d_{B_t}(v) - b_td(v))^2 \ge \sum_{v\in B_{t+1}} (d_{B_t}(v) - b_td(v))^2 \\
		& \quad \ge \left(\frac{1}{2} - b_t\right)^2 \sum_{v\in B_{t+1}} d(v)^2
		\ge \left(\frac{1}{2} - b_t\right)^2 \card{B_{t+1}} \left(\frac{1}{\card{B_{t+1}}}\sum_{v\in B_{t+1}}d(v)\right)^2.
	\end{aligned}
	\]
	Combining the three equations above, we get
	\[
	\begin{aligned}
		& \left(\frac{1}{\card{B_{t+1}}}\sum_{v\in B_{t+1}}d(v)\right)^2 \card{B_{t+1}}
		\le \|A_G - x\mathbf{1}_n\mathbf{1}_n^T\|_{op}^2 \cdot \frac{1 - b_t}{(1/2 - b_t)^2}\card{B_t}
	\end{aligned}
	\]
	Note that since $b_t \le b < 1/2$, $\frac{1 - b_t}{(1/2 - b_t)^2} \le \frac{1 - b}{(1/2 - b)^2}$.
	Passing to the infimum over $x$ completes the proof.
\end{proof}

\begin{proof}[Proof of Proposition \ref{prop:day1-lazy}]
    By a Chernoff bound, for a universal constant $C$, for each $\eps \in (0, 1)$, with probability $1 - \eps/2$, the set $\ac{V_0}$ satisfies
    \begin{equation}  \label{eq:day1-lazy-pf-bound-Vac}
    \begin{aligned}
        \ac{R_0} & = \pac \card{R_0} \pm C_\eps\sqrt{\pac\card{R_0}}, \\
        \ac{B_0} & = \pac \card{B_0} \pm C_\eps\sqrt{\pac\card{B_0}}.
    \end{aligned}
    \end{equation}
    We expose $\ac{V_0}$ such that Eq. \eqref{eq:day1-lazy-pf-bound-Vac} occurs.
    Let $m \defeq \card{\ac{V_0}}$, we have
    \begin{equation}  \label{eq:day1-lazy-pf-m-bound}
        m = \pac n \pm 2C_\eps\sqrt{\pac n} \implies \frac{1}{4}\pac n \le m \le 2\pac n.
    \end{equation}
    Set $c \defeq 10$ and $D\defeq .21$, and recall from the proof of Theorem \ref{thm:main-nonlazy} that $D_c \approx .41$ and $C_{\ref{lem:day1}}(c, D) < 14$, where $D_c$ and $C_{\ref{lem:day1}}(c, D)$ are from Lemma \ref{lem:day1}.
    
    We assume the following:
    \begin{equation}  \label{eq:day1-lazy-pf-bounds}
        p\pup\pac^2\adv \ge \max\{8KD^{-1}, 2c\}, \quad p^{1/2}\pup\pac\adv \ge \max\{24C_\eps D^{-1}, 11\eps^{-1/2}\},
    \end{equation}
    the latter also implies $Dp\pup\pac^2\adv \ge 24C_\eps\pac\sqrt{p}$, which along with the former give
    \begin{equation}  \label{eq:day1-lazy-pf-bound-1}
        Dp\pup\pac^2\adv \ge (8K + 24C_\eps\pac\sqrt{p})/2 \ge 4K + 12C_\eps\pac\sqrt{p}.
    \end{equation}
    The former also implies
    \begin{equation}  \label{eq:day1-lazy-pf-bound-2}
        p\pac\adv \ge 2c.
    \end{equation}
    We also assume that $\pac^{1/2}\adv \ge 4C_\eps n^{1/2}$, so that
    \begin{equation}  \label{eq:day1-lazy-pf-bound-adv}
        \pac\adv \ge 4C_\eps\sqrt{\pac n}.
    \end{equation}
    
    Consider $B\defeq B_1\cap \ac{V_0}$.
    We will upper-bound this set using the same strategy in the proof of Lemma \ref{lem:day1}.
    Fix arbitrary $u, v \in \ac{V_0}$, we will
    upper-bound $\E{\card{B}}$ through $\E{\one_B(v)}$, and upper-bounding $\Var{\card{B}}$ through $\Cov{\one_B(u), \one_B(v)}$.
	
	\textit{Step 1.1 - Bounding $\E{\card{B}}$.}
	We have
	\begin{equation} \label{eq:lazy-b1-formula}
		\one_{B}(v) = \one_{\upd{V_0}}(v)\one_{\hypo{B}_0}(v) + (1 - \one_{\upd{V_0}}(v))\one_{\ac{B_0}}(v).
	\end{equation}
	Before $G$ is drawn, each node is observed independently by $v$ with probability $p$, thus we can use the same expectation argument from Proposition \ref{prop:day1-exp-var} to bound $\E{\one_{\hypo{B}_0}(v)}$.
    Let $\adv' \defeq (\card{R_0} - \card{B_0})/2$. By Eq. \eqref{eq:day1-lazy-pf-bound-adv}, we have
    \begin{equation}
        \adv' \ge \pac \adv - 2C_\eps\sqrt{\pac n} \ge \pac \adv/2,
    \end{equation}
    meaning $\card{\ac{B_0}} \le m - \pac \adv/2$.
    We also have by Eq. \eqref{eq:day1-lazy-pf-bound-2},
    \begin{equation}
        c \le \min\left\{\frac{1}{2}p\pac\adv, \frac{1}{2}\sqrt{p(1 - p)\pac n}\right\}
        \le \min\left\{p\adv', \sqrt{p(1 - p)m}\right\}.
    \end{equation}
    Thus the expectation argument in Proposition \ref{prop:day1-exp-var} gives
	\[
	\E{\one_{\hypo{B}_0}(v)}
    \le \frac{1}{2} - D_c\min\Bigl\{1, \frac{p\adv'}{2\sqrt{pm}}\Bigr\}
    \le \frac{1}{2} - D_c\min\Bigl\{1, \frac{p\pac\adv}{2\sqrt{pm}}\Bigr\},
	\]
    where $D_c$ is also defined in Lemma \ref{lem:day1}.
	Since the indicators are mutually independent ($\one_{\ac{B_0}}$ is a fixed function), we have
	\[
	\begin{aligned}
		\E{\one_{B}(v)}
		& = \E{\one_{\upd{V_0}}(v)}\E{\one_{\hypo{B}_0}(v)} + (1 - \E{\one_{\upd{V_0}}(v)})\one_{\ac{B_0}}(v) \\
		& \le \pup \left(\frac{1}{2} - D_c\min\Bigl\{1, \frac{p\pac\adv}{2\sqrt{pm}}\Bigr\}\right)
		+ (1 - \pup)\one_{\ac{B_0}}(v),
	\end{aligned}
	\]
	where the inequality is from Lemma \ref{lem:day1}.
	Summing over all $v\in \ac{V_0}$, we have
	\begin{equation} \label{eq:lazy-b1-exp}
		\E{\card{B}} \le \frac{m}{2} - \pup D_c\min\left\{m, \frac{1}{2}\pac\adv\sqrt{p m}\right\}.
	\end{equation}
	
	\textit{Step 1.2 - Bounding $\Var{\card{B_1}}$.}
	Similar to the variance step in Lemma \ref{lem:day1}, the key is bounding the covariances.
    From Eq. \eqref{eq:lazy-b1-formula} and independence, we have
	\[
	\Cov{\one_{B}(u), \one_{B}(v)} = \pup^2 \Cov{\one_{\hypo{B}_0}(u), \one_{\hypo{B}_0}(v)}.
	\]
    Using the same combinatorial argument from the covariance step in Proposition \ref{prop:day1-exp-var}, we get the bound
    \[
		\Cov{\one_{\hypo{B}_0}(u), \one_{\hypo{B}_0}(v)}
		< \frac{1}{3m}.
	\]
	Summing over distinct pairs $(u, v)$ in $\ac{V_0}$ and noting $\Var{\one_B(v)}\le 1/4$, we have
	\begin{equation} \label{eq:lazy-B-var}
		\Var{\card{B}} \le \frac{m}{4} + \pup^2 m(m - 1)\frac{1}{3m} < \frac{7m}{12}
	\end{equation}
	
	\textit{Step 1.3 - Using Chebyshev.} We have by Eqs. \eqref{eq:lazy-b1-exp} and \eqref{eq:lazy-B-var}, for any $D < D_c$:
	\begin{equation}  \label{eq:day1-lazy-pf-bound-final}
	\Pr\left(\card{B} > \frac{m}{2} - D\pup \min\left\{m, \frac{1}{2}\pac \adv\sqrt{pm}\right\}\right)
	\le \frac{C_{\ref{lem:day1}}(c, D)}{\pup^2\min\{m, \frac{1}{4}p \pac^2\adv^2\}},
	\end{equation}
    with $C_{\ref{lem:day1}}(c, D)$ defined in Lemma \ref{lem:day1}.
    We have $m \ge \pac n/4$ from Eq. \eqref{eq:day1-lazy-pf-m-bound}, thus
    \[
    \begin{aligned}
        & D\pup \min\left\{m, \frac{1}{2}\pac \adv\sqrt{pm}\right\}
        \ge \frac{D}{4}\pup\pac \min\left\{n, \adv\sqrt{p\pac n} \right\} \\
        & = \frac{D}{4} \min\left\{\pup\pac^{3/2}\sqrt{pn},\ p\pup\pac^2\adv \right\}\sqrt{\frac{n}{p\pac}}
        \ge (K + 3C_\eps\pac\sqrt{p})\sqrt{\frac{n}{p\pac}},
    \end{aligned}
    \]
    where the last inequality holds by Eq. \eqref{eq:day1-lazy-pf-bound-1} with $n$ sufficiently large.
    To bound the RHS of Eq. \ref{eq:day1-lazy-pf-bound-final}, note that $C_{\ref{lem:day1}}(c, D) < 14$ and $m \ge \pac n/4$ by Eq. \eqref{eq:day1-lazy-pf-m-bound}, so
    \[
    \begin{aligned}
        & \frac{C_{\ref{lem:day1}}(c, D)}{\pup^2\min\{m, \frac{1}{4}p \pac^2\adv^2\}}
        \le \frac{56}{\pup^2\min\{\pac n, p\pac^2\adv^2\}} \\
        & = \frac{56}{\min\{\pup^2\pac n, (p^{1/2}\pup\pac\adv^2)^2\}}
        \le \frac{56}{(11\eps^{-1/2})^2} \le \frac{\eps}{2}.
    \end{aligned}
    \]
    Thus with probability at least $1 - \eps/2$, we have
    \begin{equation}  \label{eq:day1-lazy-pf-bound-B}
        \card{B} \le \frac{m}{2} -  (K + 3C_\eps\pac\sqrt{p})\sqrt{\frac{n}{p\pac}}
    \end{equation}
    Assuming this bound for $\card{B}$, we can bound $\card{B_1}$ as follows:
    \begin{equation}
    \begin{aligned}
        \card{B_1} & = \card{B} + \card{B_0} - \card{\ac{B_0}} \\
        & \le \frac{m}{2} - (K + 3C_\eps\pac\sqrt{p})\sqrt{\frac{n}{p\pac}} + (1 - \pac)\card{B_0} + C_\eps\sqrt{\pac n} \\
        & \le \pac \frac{n}{2} + (1 - \pac)\left(\frac{n}{2} - \adv\right) + 3C_\eps\sqrt{\pac n} - (K + 3C_\eps\pac\sqrt{p})\sqrt{\frac{n}{p\pac}} \\
        & \le \frac{n}{2} - K\sqrt{\frac{n}{p\pac}}.
    \end{aligned}
    \end{equation}
    Since the bound above holds when Eqs. \ref{eq:day1-lazy-pf-bound-B} and \ref{eq:day1-lazy-pf-bound-Vac} hold, it holds with probability at least $1 - \eps/2 - \eps/2 = 1 - \eps$.
    The proof is complete.
\end{proof}


\bibliographystyle{plain}
\bibliography{main}

\end{document}